\documentclass[10pt,a4paper]{amsart}
\usepackage[latin1]{inputenc}
\usepackage{amsmath}
\usepackage{amsfonts}
\usepackage{amssymb}
\usepackage{amsthm}

\usepackage{tikz}
\usepackage{subfig}

\newtheorem{thm}{Theorem}
\newtheorem{lemma}[thm]{Lemma}

\newtheorem{cor}[thm]{Corollary}
\newtheorem{prop}[thm]{Proposition}
\newtheorem{defn}{Definition}

\newcommand*\colvec[1]{\begin{pmatrix}#1\end{pmatrix}}

\usepackage{tikz}

\begin{document}

\author{Oliver J. D. Barrowclough}
\email{oliver.barrowclough@sintef.no}
\address{SINTEF ICT, Applied Mathematics, P.O. Box 124, Blindern, 0314 Oslo, Norway}
\title{A basis for the implicit representation of planar rational cubic B\'ezier curves}

\begin{abstract}
We present an approach to finding the implicit equation of a planar rational parametric cubic curve,
by defining a new basis for the representation. The basis, which contains only four cubic bivariate 
polynomials, is defined in terms of the B\'ezier control points of the curve. An explicit formula for 
the coefficients of the implicit curve is given. Moreover, these coefficients lead to simple expressions 
which describe aspects of the geometric behaviour of the curve. In particular, we present an explicit 
barycentric formula for the position of the double point, in terms of the B\'ezier control points 
of the curve. We also give conditions for when an unwanted singularity occurs in the region of interest. 
Special cases in which the method fails, such as when three of the control points are collinear, or when 
two points coincide, will be discussed separately.
\end{abstract}


\maketitle

\section{Introduction}\label{sec:intro}

Parametric curves are widely used in CAGD applications, especially in the ubiquitous rational B\'ezier and B-spline forms, due to their natural
geometric properties. It is well known that all planar rational parametric curves can be written in implicit 
form \cite{sederberg_1995}; that is, as the zero set of a single bivariate polynomial function. The availability of both the implicit and parametric 
representations, which have properties complementary to each other, is important for various applications in CAGD, such as intersection and surface 
trimming algorithms. The increase in the presence of GPUs in commodity computers has also led to renewed interest in implicit representations for 
rendering applications \cite{blinn_2005,pfeifle_2012}. Since traditionally the design phase happens using the parametric representation, a great deal 
of research has focussed on \emph{implicitization} - the conversion from the rational parametric, to the implicit form. 

Implicitization algorithms, both exact and approximate, have been investigated by many authors (e.g.,
\cite{barrowclough_2012,buse_2010,chen_2002,corless_2001,dokken_2001,sederberg_1995}). The specific case of implicitization of 
\emph{planar rational cubic curves} has also received particular attention. In \cite{sederberg_1985}, Sederberg et al.
present a method which reduces the degrees of freedom in the implicit polynomial from nine to eight by introducing 
monoid curves. In \cite{floater_1995}, Floater makes a choice of basis that allows the implicit representation to be given in terms 
of only six basis functions - a reduction from the 10 basis functions required to represent all polynomials of total degree 
three. Floater also gives explicit formulas for the coefficients and conditions for when the curve degenerates to a conic. 

Despite being motivated by methods for sparse implicitization \cite{emiris_2005}, the techniques in this paper more closely 
resemble those of Sederberg et al. \cite{sederberg_1985,sederberg_1995} and Floater \cite{floater_1995}. The method describes how, 
in most cases, the implicit form of the curve can be defined in terms of only four basis functions. The basis functions are 
constructed from the control points of a rational cubic curve given in Bernstein-B\'ezier form. In addition, the coefficients 
which describe the implicit representation are shown to contain a lot of information about the geometry of the curve. We present 
an explicit barycentric formula for the position of the double point of the curve in terms of its control points, along with criteria 
for when the double point is considered `unwanted'. The important case of degeneration to conic sections is also treated. In the 
special case when three of the control points are collinear, the method fails; potential remedies for this will also be discussed.

Although much of the paper will utilize a Cartesian system for describing and proving the results, the method is in fact independent 
of the coordinate system, and is stated in terms of purely geometric quantities. In addition, all of the formulas (except Eq. 
(\ref{eq:parasols})), including the detection and location of singularities on the curve, can be implemented in exact rational 
arithmetic; that is, only the operations of addition, subtraction multiplication and division are required. This contrasts with 
other methods, which often require polynomial rootfinding to find the double point via the parametric representation 
\cite{floater_1995,blinn_2005,pfeifle_2012,thomassen_2005}. This is potentially useful in applications which require exact precision 
and also aids the speed of floating point implementations. 

There are several applications of implicit representations of rational cubic curves. These include resolution independent curve rendering,
as in \cite{blinn_2005,pfeifle_2012}, and intersection algorithms, as in \cite{sederberg_1985,sederberg_1986,thomassen_2005}. We also
envisage great potential for the use of the techniques in this paper in surface trimming algorithms. A piecewise implicit
representation of cubic trimming curves in the parameter domain will give a simple and accurate test for whether a point lies 
inside or outside the trimming region. The geometric formulas presented in this paper are also interesting from a theoretical perspective.

The paper will proceed as follows. In Section \ref{sec:thebasis} we present the construction of the basis we use for the cubic 
implicitization, and also present formulas for the coefficients which define the curve. Computation of the double point of the curve will 
be addressed in Section \ref{sec:double}, and Section \ref{sec:conics} will cover the case of cubics which degenerate to conic sections. 
The special case of collinear control points will be discussed in Section \ref{sec:collinear}. Several examples will be presented in Section 
\ref{sec:examples}, which highlight the simplicity of the method.  We conclude the paper with a discussion of the extension to higher 
degrees and further work in Section \ref{sec:conc}. Some extended proofs and basic geometric properties are deferred to the Appendices.

\section{A basis for representing rational cubic B\'ezier curves implicitly}\label{sec:thebasis}

In this section we describe the implicit basis functions and prove some simple facts such as invariance of the coefficients under 
affine transformations. We begin with some definitions.

\subsection{Preliminaries} 

We assume that the planar rational cubic curve is given in Bernstein-B\'ezier form with control points 
$\mathbf{c}_0,\mathbf{c}_1,\mathbf{c}_2,\mathbf{c}_3,$ and weights $w_0,w_1,w_2,w_3.$ That is, we have 
\begin{equation}\label{eq:cubbez}
\mathbf{p}(t) = \frac{\mathbf{c}_0 w_0(1-t)^3 + \mathbf{c}_1 3 w_1(1-t)^2 t + \mathbf{c}_2 3 w_2(1-t)t^2 + \mathbf{c}_3 w_3t^3}
                     {w_0(1-t)^3 + 3w_1(1-t)^2 t + 3w_2(1-t)t^2 + w_3t^3}.
\end{equation}
In the parametric form, rational B\'ezier curves are normally rendered within a region of interest corresponding to the parameter values
$t\in[0,1].$ It is also common in the CAGD community to require that the weights are strictly positive. Although such a restriction is 
easier for us to work with, for the most part, it is not necessary; most of the methods we present also work with negative weights, zero weights 
and weights of mixed sign. However, due to the construction which follows, we do require that \emph{no three of the control points are 
collinear}. It is important to note that this will be an assumption of all the results up to Section \ref{sec:collinear}.

In the following definition, we assume the control points are given in a Cartesian system, $\mathbf{c}_{i} = (c_{i,0},c_{i,1}).$ 
\begin{defn}\label{def:L}
We define the implicit equation of the line between $\mathbf{c}_i$ and $\mathbf{c}_j$ to be given by $L_{ij}(x,y)=0,$ where
\[
L_{ij}(x,y) = 
\begin{vmatrix}
      x &       y & 1 \\
c_{i,0} & c_{i,1} & 1 \\
c_{j,0} & c_{j,1} & 1
\end{vmatrix}.
\]
\end{defn}
We may refer loosely to `the line $L_{ij},$' meaning `the line defined by the equation $L_{ij}(x,y)=0$'. Note that the norm of the gradient 
of $L_{ij}$ is equal to the Euclidean distance between the points $\mathbf{c}_i$ and $\mathbf{c}_j:$ \footnote{The linear functions $L_{ij},$ 
can in fact be characterised by the three conditions of vanishing at $\mathbf{c}_i$ and $\mathbf{c}_j,$ and having constant gradient, 
proportional to the Euclidean distance between the points. However, for the sake of clarity, we proceed using the definition in the Cartesian 
system.}
\[
\Vert \nabla L_{ij} \Vert_2 = \Vert \mathbf{c}_j - \mathbf{c}_i \Vert_2. 
\]

\begin{defn}\label{def:lambda}
We define a quantity $\lambda_{ijk}$  as follows:
\[
\lambda_{ijk} = 
\begin{vmatrix}
c_{i,0} & c_{i,1} & 1 \\
c_{j,0} & c_{j,1} & 1 \\
c_{k,0} & c_{k,1} & 1
\end{vmatrix}.
\]
For compactness of notation we define $\lambda_i = (-1)^{i+1}\lambda_{i+1,i+2,i+3}$ where the indices $i+1,$ $i+2$ and $i+3$ are taken 
modulo 4. That is, 
\begin{equation}
\lambda_0 = \lambda_{321}, \quad \lambda_1 = \lambda_{230}, \quad \lambda_2 = \lambda_{103}, \quad \lambda_3 = \lambda_{012}.
\end{equation}
We also make the definition 
\[
u_i = \binom{3}{i} w_i,
\]
for each of the weights $(w_i)_{i=0}^3$ of the rational cubic B\'ezier curve.
\end{defn}

Clearly $\lambda_{ijk}=0$ if any of the $i,j,k$'s are equal. 
The quantities $\lambda_i$ represent twice the signed areas of the triangles formed from the control points $(\mathbf{c}_i)_{i=0}^3,$ by omitting 
the point $\mathbf{c}_i$ (i.e., $\mathbf{c}_{i+1}, \mathbf{c}_{i+2}$ and $\mathbf{c}_{i+3}$). The areas, $\lambda_i,$ are pictured with the 
corresponding weights, $u_i,$ in Figure \ref{fig:coefs}.

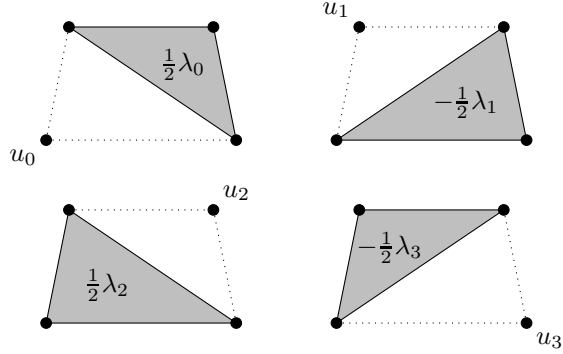
\begin{figure}
\begin{center}
\begin{tikzpicture}[scale=1.0]
\path (0,0) coordinate (A);
\path (0.3,1.5) coordinate (B);
\path (2.2,1.5) coordinate (C);
\path (2.5,0) coordinate (D);
\draw[fill=gray!50] (D) -- (B) -- (C) -- (D);
\draw[dotted] (A) -- (B);
\draw[dotted] (A) -- (D);
\draw[fill] (A) circle (0.07);
\draw[fill] (B) circle (0.07);
\draw[fill] (C) circle (0.07);
\draw[fill] (D) circle (0.07);
\draw (A) node[below left] {$u_0$};
\draw[transparent] (B) node[above left] {$u_1$};
\draw[transparent] (C) node[above right] {$u_2$};
\draw[transparent] (D) node[below right] {$u_3$};
\draw (1.8,1) node[] {$\frac{1}{2}\lambda_0$};
\end{tikzpicture} 
\begin{tikzpicture}[scale=1.0]
\draw[fill=gray!50] (A) -- (C) -- (D) -- (A);
\draw[dotted] (A) -- (B);
\draw[dotted] (B) -- (C);
\draw[fill] (A) circle (0.07);
\draw[fill] (B) circle (0.07);
\draw[fill] (C) circle (0.07);
\draw[fill] (D) circle (0.07);
\draw[transparent] (A) node[below left] {$u_0$};
\draw (B) node[above left] {$u_1$};
\draw[transparent] (C) node[above right] {$u_2$};
\draw[transparent] (D) node[below right] {$u_3$};
\draw (1.7,0.5) node[] {$-\frac{1}{2}\lambda_1$};
\end{tikzpicture} \\
\begin{tikzpicture}[scale=1.0]
\draw[fill=gray!50] (A) -- (B) -- (D) -- (A);
\draw[dotted] (B) -- (C);
\draw[dotted] (C) -- (D);
\draw[fill] (A) circle (0.07);
\draw[fill] (B) circle (0.07);
\draw[fill] (C) circle (0.07);
\draw[fill] (D) circle (0.07);
\draw[transparent] (A) node[below left] {$u_0$};
\draw[transparent] (B) node[above left] {$u_1$};
\draw (C) node[above right] {$u_2$};
\draw[transparent] (D) node[below right] {$u_3$};
\draw (0.8,0.5) node[] {$\frac{1}{2}\lambda_2$};
\end{tikzpicture} 
\begin{tikzpicture}[scale=1.0]
\draw[fill=gray!50] (A) -- (B) -- (C) -- (A);
\draw[dotted] (C) -- (D);
\draw[dotted] (A) -- (D);
\draw[fill] (A) circle (0.07);
\draw[fill] (B) circle (0.07);
\draw[fill] (C) circle (0.07);
\draw[fill] (D) circle (0.07);
\draw[transparent] (A) node[below left] {$u_0$};
\draw[transparent] (B) node[above left] {$u_1$};
\draw[transparent] (C) node[above right] {$u_2$};
\draw (D) node[below right] {$u_3$};
\draw (0.7,1) node[] {$-\frac{1}{2}\lambda_3$};
\end{tikzpicture}
\caption{The definition of the coefficients $\lambda_i$ corresponds to twice the signed areas of the shaded regions. 
The corresponding weights $u_i$ appear on the opposite vertex.}
\label{fig:coefs}
\end{center}
\end{figure}

\subsection{Implicit basis functions for rational cubic curves}

The following definition describes the basis functions we use for the implicit representation:
\begin{defn}
We define four basis functions as follows:
\begin{eqnarray*}
K_0(x,y) & = & L_{01}(x,y) L_{12}(x,y) L_{23}(x,y), \\
K_1(x,y) & = & L_{01}(x,y) L_{13}(x,y)^2, \\
K_2(x,y) & = & L_{02}(x,y)^2 L_{23}(x,y), \\
K_3(x,y) & = & L_{03}(x,y)^3.  
\end{eqnarray*}
\end{defn}
A diagrammatic representation of these basis functions is shown in Figure \ref{fig:fns}. To gain some intuition on the construction 
of these basis functions, an extended discussion is presented in \ref{sec:construct}. 

The term `basis functions' is used rather loosely in this context, since for any different collection of control points the `basis' will be different. 
The function set provides a basis for any non-degenerate planar rational cubic B\'ezier curve with given control points, subject to no three points being collinear; 
that is, for any choice of weight combinations. Moreover, the basis functions are conceptually similar for all non-collinear collections of control points, although they do differ analytically. 
The proof of linear independence given in Theorem \ref{thm:linindep}, together with the spanning property of Theorem \ref{thm:main}, 
justifies the choice of the term `basis' for this function set. 
However, the fact that the functions are dependent on the control points must be understood.

In the following Theorem we establish an explicit formula for the implicit representation of the rational cubic B\'ezier curve $\mathbf{p}(t),$ in terms of these basis functions. 

\begin{thm}\label{thm:main}
Suppose we are given a non-degenerate rational cubic B\'ezier curve, $\mathbf{p}(t),$ such that no three of the control points are collinear. Then 
the curve has the following equation defining its implicit representation:
\[
q(x,y) = \sum_{i=0}^{3} b_i K_i(x,y)
\]
where
\begin{equation}\label{eq:coefs}
\begin{split}
b_0 & = -(\lambda_1^2 \lambda_2^2 U - u_1^2 u_2^2 \Lambda), \\
b_1 & = \lambda_1^3 \lambda_3 U    - u_1^3 u_3 \Lambda, \\
b_2 & = \lambda_0 \lambda_2^3 U    - u_0 u_2^3 \Lambda, \\
b_3 & = \lambda_0^2 \lambda_3^2 U  - u_0^2 u_3^2 \Lambda.
\end{split}
\end{equation}
and $U = \prod_{k=0}^3 u_i$ and $\Lambda = \prod_{k=0}^3 \lambda_i.$

\end{thm}
Due to its length, we defer the proof of this theorem to \ref{sec:proof}. It is interesting to note that the degrees to which the 
$u_i$s and $\lambda_i$s appear in the coefficient formula (\ref{eq:coefs}), are closely related to the multiplicities of the basis functions
$(K_i)_{i=0}^3$ at the vertices $(\mathbf{c}_k)_{k=0}^3.$

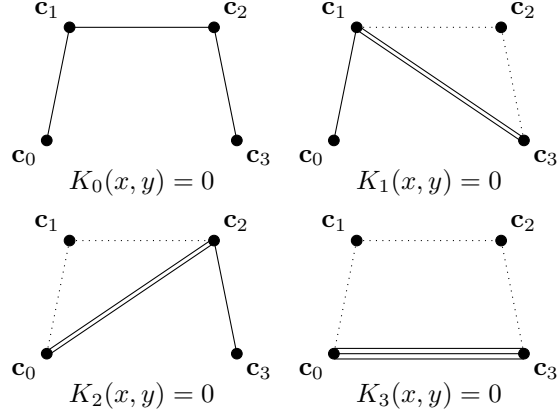
\begin{figure}
\begin{center}
\begin{tikzpicture}[scale=1.0]
\draw (A) -- (B);
\draw (B) -- (C);
\draw (C) -- (D);
\draw[fill] (A) circle (0.07);
\draw[fill] (B) circle (0.07);
\draw[fill] (C) circle (0.07);
\draw[fill] (D) circle (0.07);
\draw (A) node[below left] {$\mathbf{c}_0$};
\draw (B) node[above left] {$\mathbf{c}_1$};
\draw (C) node[above right] {$\mathbf{c}_2$};
\draw (D) node[below right] {$\mathbf{c}_3$};
\draw (1.25,-0.25) node[below] {$K_0(x,y)=0$};
\end{tikzpicture} 
\begin{tikzpicture}[scale=1.0]
\draw (A) -- (B);
\draw (0.32,1.52) -- (2.52,0.02);
\draw (0.28,1.48) -- (2.48,-0.02);
\draw[dotted] (B) -- (C);
\draw[dotted] (C) -- (D);
\draw[fill] (A) circle (0.07);
\draw[fill] (B) circle (0.07);
\draw[fill] (C) circle (0.07);
\draw[fill] (D) circle (0.07);
\draw (A) node[below left] {$\mathbf{c}_0$};
\draw (B) node[above left] {$\mathbf{c}_1$};
\draw (C) node[above right] {$\mathbf{c}_2$};
\draw (D) node[below right] {$\mathbf{c}_3$};
\draw (1.25,-0.25) node[below] {$K_1(x,y)=0$};
\end{tikzpicture} \\
\begin{tikzpicture}[scale=1.0]
\draw (0.02,-0.02) -- (2.22,1.48);
\draw (-0.02,0.02) -- (2.18,1.52);
\draw (C) -- (D);
\draw[dotted] (A) -- (B);
\draw[dotted] (B) -- (C);
\draw[fill] (A) circle (0.07);
\draw[fill] (B) circle (0.07);
\draw[fill] (C) circle (0.07);
\draw[fill] (D) circle (0.07);
\draw (A) node[below left] {$\mathbf{c}_0$};
\draw (B) node[above left] {$\mathbf{c}_1$};
\draw (C) node[above right] {$\mathbf{c}_2$};
\draw (D) node[below right] {$\mathbf{c}_3$};
\draw (1.25,-0.25) node[below] {$K_2(x,y)=0$};
\end{tikzpicture} 
\begin{tikzpicture}[scale=1.0]
\draw (0,-0.07) -- (2.5,-0.07);
\draw (A) -- (D);
\draw (0,0.07) -- (2.5,0.07);
\draw[dotted] (A) -- (B);
\draw[dotted] (B) -- (C);
\draw[dotted] (C) -- (D);
\draw[fill] (A) circle (0.07);
\draw[fill] (B) circle (0.07);
\draw[fill] (C) circle (0.07);
\draw[fill] (D) circle (0.07);
\draw (A) node[below left] {$\mathbf{c}_0$};
\draw (B) node[above left] {$\mathbf{c}_1$};
\draw (C) node[above right] {$\mathbf{c}_2$};
\draw (D) node[below right] {$\mathbf{c}_3$};
\draw (1.25,-0.25) node[below] {$K_3(x,y)=0$};
\end{tikzpicture} \\
\caption{A diagrammatic representation of the zero sets of the basis functions $(K_i)_{i=0}^3.$ The number of lines 
between any two points $\mathbf{c}_i$ and $\mathbf{c}_j,$ reflect the multiplicity with which $L_{ij}(x,y)$ appears 
in the basis function.}
\label{fig:fns}
\end{center}
\end{figure}

\begin{prop}
The coefficients $(b_i)_{i=0}^3$ defined by (\ref{eq:coefs}) are invariant under affine transformations, up to a constant scaling.
\end{prop}
\begin{proof}
Formally, suppose we are given a rational cubic B\'ezier curve $\mathbf{p}(t)$ with control points $\mathbf{c}_i$ and weights $w_i,$ whose implicit
coefficients $(b_i)_{i=0}^3$ are defined by (\ref{eq:coefs}). Then, for any affine map $\Phi,$ we claim that the implicit coefficients of 
the transformed curve $\Phi(\mathbf{p}(t))$ are given by $(\alpha b_i)_{i=0}^3,$ for some non-zero constant $\alpha.$
Since affine transformations multiply areas by a non-zero constant, there exists a constant $C$ such that $\tilde{\lambda}_i = C\lambda_i$ where 
$\tilde{\lambda}_i$ are the areas defined by control points $\tilde{\mathbf{c}}_i,$ after the affine transformation (i.e., 
$\tilde{\mathbf{c}}_i = \Phi \mathbf{c}_i$, for each $i=0,1,2,3$). Now, since the weights $u_i$ are unchanged, and the $\lambda_i$s appear 
homogeneously of degree four in the definition (\ref{eq:coefs}), we have $\tilde{b}_i = C^4 b_i,$ for each $i.$
\end{proof}
Since implicit representations are unchanged by non-zero scalar multiplication, we can clearly factor out the constant $C^4.$

\subsection{Evaluating the implicit equation}

One potential disadvantage of the method described in this paper, is that when the implicit equation is evaluated, samples are taken at the six different
lines in the set
\[
\{L_{ij}(x,y) : 0\leq i < j \leq 3\}.
\]
This can, however, be reduced to three evaluations by some simple relations between the lines. 

\begin{prop}\label{prop:l0l1l2}
Suppose we are given any four points $(\mathbf{c}_{i})_{i=0}^3,$ with no three collinear. Then, using Definitions \ref{def:L} and \ref{def:lambda}, we can write
\begin{equation}\label{eq:otherlines}
\lambda_{j}L_{ij}(x,y) + \lambda_{k}L_{ik}(x,y) + \lambda_{l}L_{il}(x,y) \equiv 0,
\end{equation}
for any choice of $i,j,k,l\in\{0,1,2,3\},$ with all indices distinct.
\end{prop}
\begin{proof}
Let 
\[
f(x,y) = \lambda_{j}L_{ij}(x,y) + \lambda_{k}L_{ik}(x,y) + \lambda_{l}L_{il}(x,y).
\]
Then 
\begin{equation*}
f(\mathbf{c}_j) = \lambda_k L_{ik}(\mathbf{c}_j) + \lambda_l L_{il}(\mathbf{c}_j) = \pm(\lambda_k \lambda_l - \lambda_l \lambda_k) = 0,
\end{equation*}
since $L_{ij}(\mathbf{c}_j) = 0,$ and $\lambda_k L_{ik}(\mathbf{c}_j)$ and $\lambda_l L_{il}(\mathbf{c}_j)$ must have opposite signs, which depend on 
the orientation of $i,j,k$ and $l.$ Similarly, $f(\mathbf{c}_k)=0$ 
and $f(\mathbf{c}_l)=0.$ Thus, since $f$ is a linear function which is zero at three non-collinear points, it must be identically zero. 
\end{proof}

This Proposition gives us an alternative method to evaluate some of the functions $L_{ij}.$ We assume that we are given the lines $L_{01},$ 
$L_{12}$ and $L_{23},$ corresponding to the lines in the control polygon.
It can easily be shown, using (\ref{eq:otherlines}), that
\begin{equation*}
\begin{split}
L_{02} & = \frac{\lambda_3 L_{23}-\lambda_1 L_{12}}{\lambda_0}, \\
L_{13} & = \frac{\lambda_0 L_{01}-\lambda_2 L_{12}}{\lambda_3}, \\
L_{03} & = \frac{\lambda_1\lambda_2 L_{12} - \lambda_0\lambda_1 L_{01} - \lambda_2\lambda_3 L_{23}}{\lambda_0\lambda_3}.
\end{split}
\end{equation*} 
When using this method as a numerical technique, care should be taken to ensure a sufficient degree of numerical stability. For example, if the 
denominators become small, it may be better to choose a different set of three lines to evaluate on, or to compute each line evaluation explicitly. 
A similar method is used by Sederberg and Parry in \cite{sederberg_1986}, in order to simplify the symbolic expansion of the determinant required 
in their method.

\subsection{Properties of the coefficients}

In the following definition we give three quantities which can be used to determine several characteristics of the geometry of the curve, such as 
when the curve degenerates, and in what region the singularity lies.

\begin{defn}\label{defn:phis}
We define three quantities $\phi_1,$ $\phi_2$ and $\phi_3$ as follows:
\begin{equation*}
\begin{split}
\phi_1 & = u_0 u_2 \lambda_1^2 - u_1^2 \lambda_0 \lambda_2, \\
\phi_2 & = u_1 u_3 \lambda_2^2 - u_2^2 \lambda_1 \lambda_3, \\
\phi_3 & = u_1 u_2 \lambda_0 \lambda_3 - u_0 u_3 \lambda_1 \lambda_2.  
\end{split}
\end{equation*}
\end{defn}

These quantities are based on the coefficients $(b_i)_{i=0}^3$ and we can write 
\begin{equation}\label{eq:phicoefs}
\begin{split}
b_0 & = \phi_3 u_1 u_2 \lambda_1 \lambda_2, \\
b_1 & = \phi_1 u_1 u_3 \lambda_1 \lambda_3, \\
b_2 & = \phi_2 u_0 u_2 \lambda_0 \lambda_2, \\
b_3 & = \phi_3 u_0 u_3 \lambda_0 \lambda_3.
\end{split}
\end{equation}
We thus have a compact form of the implicit equation
\begin{equation}\label{eq:qexplicit}
\begin{split}
q(x,y) =  \phi_3 & (u_1 u_2 \lambda_1 \lambda_2 K_0(x,y) + u_0 u_3 \lambda_0 \lambda_3 K_3(x,y)) \\ 
         & + \phi_1 u_1 u_3 \lambda_1 \lambda_3 K_1(x,y) + \phi_2 u_0 u_2 \lambda_0 \lambda_2 K_2(x,y).
\end{split}
\end{equation}
There is also a relation between the three quantities given as follows:
\begin{prop}\label{prop:p1p2p3}
The following two equations hold:
\[
u_3\lambda_2\phi_1 + u_1\lambda_0\phi_2 + u_2\lambda_1\phi_3 = 0,
\]
and 
\[
u_2\lambda_3\phi_1 + u_0\lambda_1\phi_2 + u_1\lambda_2\phi_3 = 0.
\]
\end{prop}
\begin{proof}
The proof is a simple exercise in expanding polynomials after substituting the definitions $\phi_1,\phi_2$ and $\phi_3.$ We
therefore omit the details here.
\end{proof}

\section{Double points on rational cubic curves}\label{sec:double}

Since the curve is rational, there always exists a single double point in the form of a crunode, an acnode or a cusp. It is a surprising fact that 
the double point of a rational cubic curve with rational coefficients is necessarily rational, despite the fact that the corresponding parameter 
values may be irrational or complex. This was apparently first noticed by Sederberg in \cite{sederberg_1985} (Theorem 1):
\begin{center}
\emph{``Every rational cubic curve defined by polynomials with rational coefficients has a double point whose coordinates are real and rational''.}
\end{center}

In this section we derive the equations of two lines which intersect at the double point, in terms of the coefficients we have already discussed. We 
give exact formulas for the location of the double point in terms of barycentric combinations of its control points. We also define a condition which 
detects when there is an `unwanted' self-intersection in the region of interest.

We first give some identities for the gradient which will be needed in the proofs in this section. From Definition \ref{def:L} it is clear that we have  
\[
\frac{\partial}{\partial x} L_{ij}(x,y) = c_{i1} - c_{j1},  
\]
and
\[
\frac{\partial}{\partial y} L_{ij}(x,y) = c_{j0} - c_{i0}. 
\]
For compactness of notation we define $\mathbf{c}_{ij}^{\bot} = \colvec{c_{i1}-c_{j1} \\ c_{j0}-c_{i0}}.$ Thus, using the product rule, we can write the 
gradients of the basis functions as follows:
\begin{equation}\label{eq:gradki}
\begin{split}
\nabla K_{0}(x,y) & = \mathbf{c}_{23}^{\bot} L_{01} L_{12} +
                      \mathbf{c}_{12}^{\bot} L_{01} L_{23} +
                      \mathbf{c}_{01}^{\bot} L_{12} L_{23}, \\
\nabla K_{1}(x,y) & = \mathbf{c}_{13}^{\bot} 2 L_{01} L_{13} + 
                      \mathbf{c}_{01}^{\bot} L_{13}^2, \\
\nabla K_{2}(x,y) & = \mathbf{c}_{02}^{\bot} 2 L_{02} L_{23} + 
                      \mathbf{c}_{23}^{\bot} L_{02}^2, \\
\nabla K_{3}(x,y) & = \mathbf{c}_{03}^{\bot} 3 L_{03}^2.
\end{split}
\end{equation}
In addition, using (\ref{eq:qexplicit}), we can write the gradient function, $\nabla q,$ in the compact form
\begin{equation}\label{eq:gradq}
\begin{split}
\nabla q(x,y) =  \phi_3 & (u_1 u_2 \lambda_1 \lambda_2 \nabla K_0(x,y) + u_0 u_3 \lambda_0 \lambda_3 \nabla K_3(x,y)) \\ 
                & + \phi_1 u_1 u_3 \lambda_1 \lambda_3 \nabla K_1(x,y) + \phi_2 u_0 u_2 \lambda_0 \lambda_2 \nabla K_2(x,y).
\end{split}
\end{equation}

\subsection{Location of the singularity}

In this section we determine the location of the singularity in affine space. 
The following proposition deals with the case where the singularity occurs at one of the endpoints $\mathbf{c}_0$ or $\mathbf{c}_3.$ 

\begin{prop}\label{prop:endpoints}
Let the control points and weights of a non-degenerate rational cubic B\'ezier curve be given, and assume that no three control points are collinear. 
Then the singularity occurs at the end point $\mathbf{c}_0$ if and only if $\phi_2=0$ or $u_1=0.$ Similarly, the singularity occurs at the end point 
$\mathbf{c}_3$ if and only if $\phi_1=0$ or $u_2=0.$
\end{prop}
\begin{proof}
Since the functions $(L_{0i})_{i=1}^3,$ evaluated at the endpoint $\mathbf{c}_0$ are all zero, we have $\nabla K_{2}(\mathbf{c}_0)=\nabla K_{3}(\mathbf{c}_0)=0.$ Thus
\begin{equation*}
\begin{split}
\nabla q(\mathbf{c}_0)     & = \mathbf{c}_{01}^{\bot} \left( L_{12}(\mathbf{c}_0)L_{23}(\mathbf{c}_0) \left( u_1^2 u_2^2 \Lambda - \lambda_1^2\lambda_2^2 U \right) 
                                                           + L_{13}^2(\mathbf{c}_0)\left(\lambda_1^3\lambda_3 U - u_1^3 u_3 \Lambda \right) \right) \\
                           & = \mathbf{c}_{01}^{\bot} \Lambda u_1^2 \left( u_2^2 \lambda_1\lambda_3 - u_1 u_3 \lambda_2^2 \right) \\
                           & = -\mathbf{c}_{01}^{\bot} \Lambda u_1^2 \phi_2.
\end{split}
\end{equation*}
Clearly this is zero if and only if $\phi_2=0$ or $u_1=0,$ since $\mathbf{c}_0$ and $\mathbf{c}_1$ are distinct. Similarly, $\nabla q(\mathbf{c}_3)=0$ if and 
only if $\phi_1=0$ or $u_2=0.$
\end{proof}

In non-degenerate cases, except for those with the conditions above, we can derive the equations of two lines, $\tilde{S}_1$ and $\tilde{S}_2,$ which both intersect 
the singularity and either $\mathbf{c}_0$ or $\mathbf{c}_3.$

\begin{prop}\label{prop:singlines}
Suppose we have a planar rational cubic B\'ezier curve with no three control points collinear and non-zero weights. Suppose further that the double point of 
the curve is finite, and does not lie on either of the endpoints $\mathbf{c}_0$ or $\mathbf{c}_3.$ Define two equations as follows:
\begin{equation}\label{eq:singlines1}
\begin{split}
\tilde{S}_1(x,y) & = L_{02}(x,y) u_2 \phi_1 - L_{03}(x,y)u_1 \phi_3 = 0, \\ 
\tilde{S}_2(x,y) & = L_{13}(x,y) u_1 \phi_2 - L_{03}(x,y)u_2 \phi_3 = 0.
\end{split}
\end{equation}
The equations (\ref{eq:singlines1}) define two lines which intersect the end points of the B\'ezier curve, $\mathbf{c}_0$ and 
$\mathbf{c}_3$ respectively. Moreover, the lines intersect each other at the unique double point of the curve. The lines are also determined by the 
alternative equations 
\begin{equation}\label{eq:singlines2}
\begin{split}
\hat{S}_1(x,y) & = L_{01}(x,y) u_2 \phi_1 - L_{03}(x,y)u_0 \phi_2  = 0, \\
\hat{S}_2(x,y) & = L_{23}(x,y) u_1 \phi_2 - L_{03}(x,y)u_3 \phi_1 = 0.
\end{split}
\end{equation}
\end{prop}
\begin{proof}
We can immediately see that $\tilde{S}_1(\mathbf{c}_0)=0$ since, by definition, it is a linear combination of $L_{02}$ and $L_{03}.$ Similarly, 
$\tilde{S}_2(\mathbf{c}_3)=0.$ It remains to show that both lines intersect at the double point of $q.$ 

We now use both equations (\ref{eq:singlines1}) and (\ref{eq:singlines2}), which can easily be shown to be equivalent, using Proposition \ref{prop:l0l1l2}. 
We must prove that both the implicit polynomial $q$ and its gradient $\nabla q,$ vanish when these equations are satisfied. From (\ref{eq:singlines1}) and 
(\ref{eq:singlines2}) we infer that 
\begin{equation}\label{eq:l03cond}
L_{01} = \frac{L_{03}u_0 \phi_2}{u_2 \phi_1}, \
L_{23} = \frac{L_{03}u_3 \phi_1}{u_1 \phi_2}, \
L_{02} = \frac{L_{03}u_1 \phi_3}{u_2 \phi_1}, \
L_{13} = \frac{L_{03}u_2 \phi_3}{u_1 \phi_2}, \
\end{equation}
and
\begin{equation}
L_{12} = \frac{L_{23} \lambda_3-L_{02}\lambda_0}{\lambda_1} =
                                  L_{03}\left(\frac{\phi_1 u_3 \lambda_3}{\phi_2 u_1 \lambda_1}-\frac{\phi_3 u_1 \lambda_0}{\phi_1 u_2 \lambda_1}\right).
\end{equation}
These are well defined since, by Proposition \ref{prop:endpoints}, we have $\phi_1\neq0$ and $\phi_2\neq0,$ and the weights are non-zero. Thus, by 
(\ref{eq:qexplicit}), the formula for $q$ becomes
\begin{equation*}
\begin{split}
q(x,y) & = \phi_3 \lambda_0 \lambda_3 u_0 u_3 L_{03}^3  + \phi_3\left(\frac{\phi_1 \lambda_3 u_3}{\phi_2 u_1 \lambda_1}
                                    -  \frac{\phi_3 \lambda_0 u_1}{\phi_1 u_2 \lambda_1}\right) \lambda_1\lambda_2 u_0 u_3 L_{03}^3 \\
       &  \quad\quad\quad\quad\quad\quad\quad\quad\quad\quad\quad
        + \frac{\phi_3^2 \lambda_1 \lambda_3 u_0 u_2 u_3}{\phi_2 u_1}L_{03}^3  + \frac{\phi_3^2 \lambda_0 \lambda_2 u_0 u_1 u_3}{\phi_1 u_2}L_{03}^3,  \\
       & = L_{03}^{3} \frac{\lambda_3 \phi_3 u_0 u_3}{\phi_2 u_1} \left(\phi_2 \lambda_0 u_1 + \phi_3 \lambda_1 u_2 + \phi_1 \lambda_2 u_3\right), \\  
       & = 0,
\end{split}
\end{equation*}
by Proposition \ref{prop:p1p2p3}. 

In order to prove that $\nabla q = 0$ when $S_1=S_2=0,$ we substitute the equations (\ref{eq:l03cond}) into the gradient identities (\ref{eq:gradki}). 
After expanding the formula (\ref{eq:gradq}), we can group together the coefficients of $\mathbf{c}_{ij}^\perp$ to write
\[
\nabla q = L_{03}^2 \left(\psi_{01}\mathbf{c}_{01}^\perp+\psi_{12}\mathbf{c}_{12}^\perp+\psi_{23}\mathbf{c}_{23}^\perp
                           +\psi_{02}\mathbf{c}_{02}^\perp+\psi_{13}\mathbf{c}_{13}^\perp+\psi_{03}\mathbf{c}_{03}^\perp\right),
\]
where 
\begin{equation*}
\begin{split}
\psi_{01} & = \phi_3 u_0 u_3 \lambda_0\lambda_1, \\
\psi_{12} & = \phi_3 u_0 u_3 \lambda_1\lambda_2, \\
\psi_{23} & = \phi_3 u_0 u_3 \lambda_2\lambda_3, \\
\psi_{02} & = 2\phi_3 u_0 u_3 \lambda_0\lambda_2,  \\
\psi_{13} & = 2\phi_3 u_0 u_3 \lambda_1\lambda_3,  \\
\psi_{03} & = 3\phi_3 u_0 u_3 \lambda_0\lambda_3. 
\end{split}
\end{equation*}
The quantities $\psi_{ij}$ can be verified by writing out the coefficients. Thus, we can factor out the common factor $\phi_3 u_0 u_3$, to get
\[
\nabla q = L_{03}^2 \phi_3 u_0 u_3 \left(\lambda_0\lambda_1\mathbf{c}_{01}^\perp+\lambda_1\lambda_2\mathbf{c}_{12}^\perp+\lambda_2\lambda_3\mathbf{c}_{23}^\perp
                              +2\lambda_0\lambda_2\mathbf{c}_{02}^\perp+2\lambda_1\lambda_3\mathbf{c}_{13}^\perp+3\lambda_0\lambda_3\mathbf{c}_{03}^\perp\right).
\]
Now, we can individually verify each of the following
\begin{equation*}
\begin{split}
\lambda_0\lambda_1\mathbf{c}_{01}^\perp + \lambda_0\lambda_2\mathbf{c}_{02}^\perp + \lambda_0\lambda_3\mathbf{c}_{03}^\perp  & = 0, \\
\lambda_2\lambda_3\mathbf{c}_{23}^\perp + \lambda_1\lambda_3\mathbf{c}_{13}^\perp + \lambda_0\lambda_3\mathbf{c}_{03}^\perp & = 0,  \\
\lambda_1\lambda_2\mathbf{c}_{12}^\perp + \lambda_0\lambda_2\mathbf{c}_{02}^\perp 
                                             + \lambda_1\lambda_3\mathbf{c}_{13}^\perp + \lambda_0\lambda_3\mathbf{c}_{03}^\perp & = 0.
\end{split}
\end{equation*}
Thus, summing the previous three equations, we must have $\nabla q = 0,$ proving the statement.
\end{proof}

The lines defined in Proposition \ref{prop:singlines} are shown in Figure \ref{fig:singlines}. It can be seen that the segments of the curve limited to the 
quadrants defined by the lines $\tilde{S}_1$ and $\tilde{S}_2,$ are non-singular. This is necessarily true due to rational cubic curves having only one singularity. 
The following Theorem, which is a consequence of the previous Proposition, gives the location of the singularity in the affine plane.

\begin{thm}\label{thm:doublepoint}
Suppose that the double point $\mathbf{s},$ of a non-degenerate rational cubic curve, is not at infinity. Then the following barycentric combination of the four B\'eizer control 
points $(\mathbf{c}_i)_{i=0}^3$ determines its location exactly:
\[
\mathbf{s} = \frac{
                \mathbf{c}_0 \lambda_1\phi_1\phi_3 u_2^2 
              + \frac{1}{2}(\mathbf{c}_0\lambda_0-\mathbf{c}_1\lambda_1+\mathbf{c}_2\lambda_2-\mathbf{c}_3\lambda_3)\phi_1\phi_2 u_1 u_2 
              - \mathbf{c}_3 \lambda_2\phi_2\phi_3 u_1^2 }
             { \lambda_1\phi_1\phi_3 u_2^2 
              + \frac{1}{2}(\lambda_0-\lambda_1+\lambda_2-\lambda_3)\phi_1\phi_2 u_1 u_2 
              - \lambda_2\phi_2\phi_3 u_1^2 } 
\]
The formula can be simplified in a number of different ways, using the identities provided in the previous sections. For example, 
in the barycentric coordinate system defined by $\mathbf{c}_0,$ $\mathbf{c}_2$ and $\mathbf{c}_3,$ we have
\begin{equation}\label{eq:sreduced1}
\mathbf{s} = \frac{ \mathbf{c}_0 \phi_1^2 u_2 u_3 - \mathbf{c}_2 \phi_1\phi_2 u_1 u_2 + \mathbf{c}_3 \phi_2\phi_3 u_1^2 }
              {\phi_1^2 u_2 u_3 - \phi_1\phi_2 u_1 u_2 + \phi_2\phi_3 u_1^2},
\end{equation}
or in the barycentric coordinate system defined by $\mathbf{c}_0,$ $\mathbf{c}_1$ and $\mathbf{c}_3$ 
\begin{equation}\label{eq:sreduced2}
\mathbf{s} = \frac{ \mathbf{c}_0 \phi_1\phi_3 u_2^2 - \mathbf{c}_1 \phi_1\phi_2 u_1 u_2 + \mathbf{c}_3 \phi_2^2 u_1 u_0 }
              {\phi_1\phi_3 u_2^2 - \phi_1\phi_2 u_1 u_2 + \phi_2^2 u_1 u_0}.
\end{equation}
\end{thm}
\begin{proof}
The formulas can be found by solving the linear equations $\tilde{S}_1(x,y) = 0$ and $\tilde{S}_2(x,y) = 0,$ in a particular coordinate system. It is clear that all
the formulas are barycentric combinations of the control points, by observing the coefficients in the numerator and denominator.
In order to validate that $\mathbf{s}$ is the intersection of the two lines, we can simply evaluate the functions $\tilde{S}_1$ and $\tilde{S}_2$ at $\mathbf{s}:$
\begin{equation*}
\begin{split}
\tilde{S}_1(\mathbf{s}) & = \frac{u_1\phi_2}{\phi_1^2 u_2 u_3 - \phi_1\phi_2 u_1 u_2 + \phi_2\phi_3 u_1^2} 
                               \left( - \tilde{S}_1(\mathbf{c}_2)\phi_1 u_2 - \tilde{S}_1(\mathbf{c}_3)\phi_3 u_1\right) \\
                        & = \frac{u_1\phi_2}{\phi_1^2 u_2 u_3 - \phi_1\phi_2 u_1 u_2 + \phi_2\phi_3 u_1^2}  
                              \left(- L_{03}(\mathbf{c}_2)\phi_1\phi_3 u_1 u_2 - L_{02}(\mathbf{c}_3)\phi_1\phi_3 u_1 u_2  \right) \\
                        & = 0.
\end{split}
\end{equation*}  
We can similarly show that $\tilde{S}_2$ vanishes at $\mathbf{s},$ and thus, by Proposition \ref{prop:singlines}, $\mathbf{s}$ is the unique double point of the curve.
\end{proof}
Of course, describing the singularity in terms of the control points may not be optimal with respect to numerical stability if the double point lies far from the region of
interest, due to the denominator becoming small. However, information about the singularity is normally required only when it interferes with the region of interest of the 
curve. In these cases, the method performs very well.

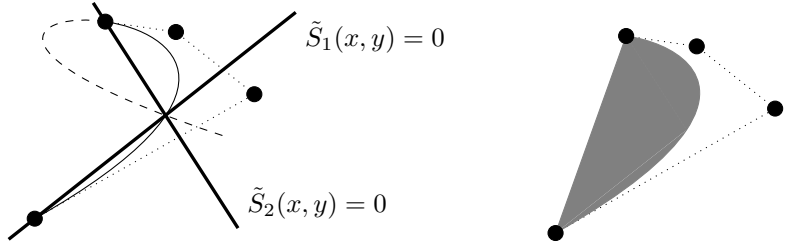
\begin{figure}
\begin{center}
       \begin{tikzpicture}[scale=3.3]
      \path (17/32, 19/24) coordinate (A0);
      \path (1/4, 5/6) coordinate (B0);
      \path (0, 2/3) coordinate (C0);
      \path (1, 1/3) coordinate (D0);
      \path (1/4, 0) coordinate (A1);
      \path (9/8, 1/2) coordinate (B1);
      \path (13/16, 3/4) coordinate (C1);
      \path (17/32, 19/24) coordinate (D1);
      \path (0.772, 0.414) coordinate (S);
      \path (0.146, -0.083) coordinate (A2);
      \path (1.29, 0.829) coordinate (SA2);
      \path (0.483, 0.867) coordinate (D2);
      \path (1.06, -0.038) coordinate (SD2);
      \draw[dotted] (A1) -- (B1);
      \draw[dotted] (B1) -- (C1);
      \draw[dotted] (C1) -- (D1);
      \draw[line width=5/4pt] (A2) -- (SA2);
      \draw[line width=5/4pt] (D2) -- (SD2);
      \draw[dashed] (A0) .. controls (B0) and (C0) .. (D0);
      \draw (A1) .. controls (B1) and (C1) .. (D1);
      \draw[fill] (A1) circle (0.03);
      \draw[fill] (B1) circle (0.03);
      \draw[fill] (C1) circle (0.03);
      \draw[fill] (D1) circle (0.03);
      \draw (SA2) node[below right] {$\tilde{S}_1(x,y)=0$};
      \draw (SD2) node[above right] {$\tilde{S}_2(x,y)=0$};
      \end{tikzpicture} \hspace*{1cm}
      \begin{tikzpicture}[scale=3.3]
      \path (1/4, 0) coordinate (A0);
      \path (9/8, 1/2) coordinate (B0);
      \path (13/16, 3/4) coordinate (C0);
      \path (17/32, 19/24) coordinate (D0);
      \path (1/4, 0) coordinate (A1);
      \path (0.538466978587320, 0.164838273478468) coordinate (B1);
      \path (0.697868589258669, 0.302504890553575) coordinate (C1);
      \path (0.771874389072580, 0.414492827535012) coordinate (D1);
      \path (0.771874389072580, 0.414492827535012) coordinate (A2);
      \path (0.922348615202335, 0.642195226475667) coordinate (B2);
      \path (0.719778471168361, 0.763736522789872) coordinate (C2);
      \path (17/32, 19/24) coordinate (D2);
      \path (0.771874389072580, 0.414492827535012) coordinate (S);
      \draw[dotted] (A0) -- (B0);
      \draw[dotted] (B0) -- (C0);
      \draw[dotted] (C0) -- (D0);
      \draw[fill,color=gray] (A0) -- (D0) -- (S);
      \draw[fill,color=gray] (A1) .. controls (B1) and (C1) .. (D1);
      \draw[fill,color=gray] (A2) .. controls (B2) and (C2) .. (D2);
      \draw[fill] (A0) circle (0.03);
      \draw[fill] (B0) circle (0.03);
      \draw[fill] (C0) circle (0.03);
      \draw[fill] (D0) circle (0.03);
      \end{tikzpicture} 
   \caption{Two lines $\tilde{S}_1$ and $\tilde{S}_2$ (in bold, left), each defined by a linear combination of $L_{ij}$s, which intersect at the singularity. For an unwanted
   self-intersection (as pictured), the segments of the curves restricted to the quadrants defined by the two lines can be rendered separately, thus visually eliminating 
   the singularity (right).}
   \label{fig:singlines}
\end{center}
\end{figure}

\subsection{Parametric identities}

In complement to the previous section, it is also possible to derive formulas for the parameter values of the singularity in terms of the quantities $\phi_1, \phi_2$ and 
$\phi_3.$ For compactness of notation we first make the following definitions:
\begin{defn}
\[
\Phi_1 = \phi_1 u_2 u_3, \quad \Phi_2 = \phi_2 u_0 u_1, \quad \Phi_3 = \phi_3 u_1 u_2.
\]
\end{defn}

\begin{prop}
Suppose we are given a non-degenerate rational cubic parametric curve with no three control points collinear. Then the parameter values of the double point are given by the 
solutions to the quadratic equation $r(t)=0,$ where 
\begin{equation}\label{eq:r}
r(t) = \Phi_1  t^2 + \Phi_3 t (1-t) + \Phi_2 (1-t)^2.
\end{equation}
\end{prop}
\begin{proof}By inserting the parametric form $\mathbf{p}(t),$ into $\tilde{S}_1$ (and factoring out the denominator), we obtain a cubic polynomial whose three roots correspond 
to the two parameter values $t_1$ and $t_2$ of the double point, and the parameter $t=0.$ After factoring out $t$ from the polynomial we obtain 
$r(t)= \tilde{S}_1(\mathbf{p}(t))/t,$ as given above. In the case that $\phi_2=0,$ $\tilde{S}_1$ degenerates, however, we can simply follow a similar proof, using $\tilde{S}_2$ 
instead.
\end{proof}

The polynomial $r(t)$ is strictly quadratic except in the case when $\Phi_1+\Phi_2-\Phi_3=0.$ In the non-degenerate case, this condition corresponds to when at least one of the parameter values of the singularity is infinite. When $r(t)$ is quadratic, it can be solved explicitly, to give the parameters $t_1$ and $t_2$ of the double point as
\begin{equation}\label{eq:parasols}
\frac{2 \Phi_2 - \Phi_3 \pm \sqrt{ \Phi_3^2 - 4 \Phi_1 \Phi_2}}{2 (\Phi_1 + \Phi_2 - \Phi_3)}.
\end{equation}
In particular, we have the discriminant $\Delta = \Phi_3^2 - 4 \Phi_1 \Phi_2,$ which defines whether the curve has a self-intersection 
($\Delta>0$), a cusp ($\Delta=0$), or an acnode ($\Delta<0$). 
This appears to be similar to the discriminant described by Stone and DeRose in \cite{stone_1989}. 

The lack of symmetry in Formula (\ref{eq:parasols}) above is due to Bernstein polynomials being defined over the interval $[0,1],$ as opposed to an interval which is symmetric about $0,$ such as $[-1,1].$

\subsection{Detecting Unwanted Self-intersections}\label{ssec:unwanted}

The parameter values of the double point can occur in several ways. If the two parameters are real and distinct, then a self-intersection occurs; in the case that the 
two parameters are equal, we have a cusp; and parameter pairs which occur as complex conjugates give rise to isolated singular points, or acnodes. 

When the curve is given in rational B\'ezier form, the region of interest is the parameter interval $[0,1].$ If the parameters of the singularity both lie within the interval 
$[0,1],$ the self-intersection is normally an intended product of the designer. If the parameters both lie outside the interval $[0,1],$ this means there will be no 
singularities in the region of interest; again, this is normally intended by the designer. The case where one parameter lies within the interval and one lies outside 
is what we term an \emph{unwanted self-intersection} or \emph{unwanted singularity}. Figure \ref{fig:unwanted} shows the three cases. When the parametric representation 
is used the unwanted case is not normally distinguished, since the curve is only plotted in the region of interest in the parameter domain. However, using implicit 
representations, it is more difficult to avoid plotting the curve without the unwanted branch and self-intersection, since a 2D region of rendering must be chosen.

\begin{figure}
 \begin{center}
      \subfloat{
      \begin{tikzpicture}[scale=1.5]
      \path (1, 0) coordinate (A0);
      \path (5/2, 3/2)coordinate (B0);
      \path (0, 3/2) coordinate (C0);
      \path (3/2, 0) coordinate (D0);
      \draw[dotted] (A0) -- (B0);
      \draw[dotted] (B0) -- (C0);
      \draw[dotted] (C0) -- (D0);
      \draw (A0) .. controls (B0) and (C0) .. (D0);
      \draw[fill] (A0) circle (0.05);
      \draw[fill] (B0) circle (0.05);
      \draw[fill] (C0) circle (0.05);
      \draw[fill] (D0) circle (0.05);
      \draw (5/4,-1/5) node[below] {(a) Singular segment};
      \end{tikzpicture}}
      \subfloat{
      \begin{tikzpicture}[scale=1.5]
      \path (371/250, 18/25) coordinate (A2);
      \path (202/125, 63/50) coordinate (B2);
      \path (221/250, 63/50) coordinate (C2);
      \path (127/125, 18/25) coordinate (D2);
      \draw[dotted] (A2) -- (B2);
      \draw[dotted] (B2) -- (C2);
      \draw[dotted] (C2) -- (D2);
      \draw[dashed] (A0) .. controls (B0) and (C0) .. (D0);
      \draw (A2) .. controls (B2) and (C2) .. (D2);
      \draw[fill] (A2) circle (0.05);
      \draw[fill] (B2) circle (0.05);
      \draw[fill] (C2) circle (0.05);
      \draw[fill] (D2) circle (0.05);
      \draw (5/4,-1/5) node[below] {(b) Non-singular segment};
      \end{tikzpicture}} 
      \subfloat{
      \begin{tikzpicture}[scale=1.5]
      \path (1, 0) coordinate (A3);
      \path (7/4, 3/4) coordinate (B3);
      \path (3/2, 9/8) coordinate (C3);
      \path (5/4, 9/8) coordinate (D3);
      \draw[dotted] (A3) -- (B3);
      \draw[dotted] (B3) -- (C3);
      \draw[dotted] (C3) -- (D3);
      \draw[dashed] (A0) .. controls (B0) and (C0) .. (D0);
      \draw (A3) .. controls (B3) and (C3) .. (D3);
      \draw[fill] (A3) circle (0.05);
      \draw[fill] (B3) circle (0.05);
      \draw[fill] (C3) circle (0.05);
      \draw[fill] (D3) circle (0.05);
      \draw (5/4,-1/5) node[below] {(c) Unwanted singularity};
      \end{tikzpicture}}
 \end{center}
 \caption{Three B\'ezier representations of the same rational cubic curve. The solid part of the curve represents the region of interest and dashed part 
 corresponds to parameter values outside $[0,1].$}
 \label{fig:unwanted}
\end{figure}
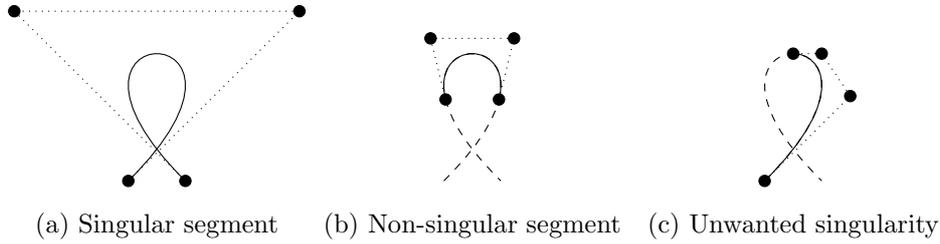

It is possible to detect the presence of unwanted singularities by directly analysing the coefficients $\Phi_1$ and $\Phi_2.$ In \cite{floater_1995}, a method is presented 
to detect unwanted singularities in special cases, which essentially correspond rational cubic B\'ezier curves which form simple arches \cite{pfeifle_2012}. The following 
proposition gives a condition to detect unwanted singularities in \emph{all} non-degenerate cases. We make the condition that a singularity with infinite parameter values is not classified as unwanted so that we can assume that $\Phi_1+\Phi_2-\Phi_3\neq0.$

\begin{prop}\label{prop:unwanted}
Let the control points and non-zero weights of a non-degenerate rational cubic B\'ezier curve be given, such that no three control points are collinear. Then there 
exists an unwanted singularity in the region of interest if and only if $\Phi_1\Phi_2 < 0.$
\end{prop}
\begin{proof}
This can be proved by observing that $r(t)$ is a quadratic polynomial in Bernstein form, with coefficients $\Phi_1,$ $\Phi_3/2$ and $\Phi_2,$ and by the properties of 
Bernstein polynomials. 

($\impliedby$) Assume first that $\Phi_1$ and $\Phi_2$ have opposite signs. Then, by observing the discriminant, we always 
have two real roots. Since $r(0)=\Phi_2$ and $r(1) = \Phi_1,$ we know there is a root in $[0,1]$ by the intermediate value theorem. We also know that there must be a 
root outside $[0,1],$ because, since it is quadratic, $r(t)$ must have the same sign as $r(-t)$ asymptotically. Thus an unwanted self-intersection occurs (c.f. 
Figure \ref{fig:unwanted}(c)).

Assume now that $\Phi_1$ and $\Phi_2$ have the same sign. Since we are only interested in cases of self-intersection, where we have two distinct real roots, 
we consider the two cases, $\Phi_3>\sqrt{4\Phi_1\Phi_2}$ and $\Phi_3<-\sqrt{4\Phi_1\Phi_2}.$ In one of the cases, all of the Bernstein coefficients of $r(t)$ have the 
same sign, so, by the variation diminishing property, both roots must be outside $[0,1].$ In the other case, $\Phi_3$ has opposite sign to $\Phi_1$ and $\Phi_2,$ so the 
derivatives of $r$ at $t=0$ and $t=1$ must have opposite sign. Again, by the intermediate value theorem, $r'(t)=0$ for some $t\in[0,1],$ thus both roots of $r(t)$ must 
be in $[0,1].$ Thus either we have a singular or non-singular segment (c.f. Figure \ref{fig:unwanted}(a) and (b)).

($\implies$) Assume first that $r(t)$ has one root in $[0,1]$ and one outside $[0,1].$ Then clearly, $r(0)=\Phi_2$ must have opposite sign to $r(1)=\Phi_1.$ 

Assume that both roots are in $[0,1].$ Then $r(0)=\Phi_2$ must have the same sign as $r(1)=\Phi_1.$ This is similarly the case if both roots are outside $[0,1].$
\end{proof}

Of course, in the standard case that all weights are positive, we can use the simplified condition $\phi_1\phi_2<0$ in Proposition \ref{prop:unwanted}. 

In \cite{blinn_2005}, Loop and Blinn use subdivision at the parameter values of the singularity in order to remove the unwanted branch. Their method involves solving a quadratic polynomial for the parameter values of the self-intersection and then running the de Casteljau algorithm for the subdivision. The results of this section suggest an alternative method to remove the singularity. If, according to Proposition \ref{prop:unwanted}, we detect an unwanted singularity, we can use the lines defined in Proposition 
\ref{prop:singlines} to render non-singular segments of the curve. Such an approach is pictured in Figure \ref{fig:singlines} (right). This would appear to be advantageous since 
we avoid the necessity of dealing with two separate curves, and also avoid computing the parameter values of the singularity. Additionally, exact rational arithmetic can be used since we do not need to solve any polynomial equations.

\section{Degeneration to conic sections}\label{sec:conics}

\subsection{Conditions for conic degeneration}

In this section we describe necessary and sufficient conditions for the degeneration of the rational cubic curve to a conic section. Some of the results of this section are similar to those of Wang and Wang \cite{wang_1992} and we use Theorem 1 from that paper in the proof of Theorem \ref{thm:conic} below\footnote{In the notation of \cite{wang_1992} the quantities $(S_i)_{i=0}^3$ are equivalent to the quantities $(\lambda_i)_{i=0}^3$ in this paper. In addition, it should be noted that while the result of Theorem 2 in \cite{wang_1992} is restricted to strictly positive weights, the results of Theorems 1 and 3 are general for any non-zero weights.}. 
We first state a Lemma which holds for any rational cubic B\'ezier curve $\mathbf{p},$ including when it degenerates to 
a conic.

\begin{lemma}\label{lem:degen}
Let $q_2(x,y) = u_0 u_3 L_{03}(x,y)^2 - u_1 u_2 L_{01}(x,y)L_{23}(x,y).$ Then for any rational cubic B\'ezier curve $\mathbf{p}(t),$ given by (\ref{eq:cubbez}), we have 
\[
q_2(\mathbf{p}(t)) = \frac{t^2(1-t)^2}{w(t)^2} r(t),
\]
where $r(t)$ is given by (\ref{eq:r}).
\end{lemma}
We defer the proof of this lemma to \ref{sec:proof}. The following Theorem defines necessary and sufficient conditions for conic degeneration of the rational cubic curve.
\begin{thm}\label{thm:conic}
Suppose all the weights $(w_i)_{i=0}^3$ are non-zero and no three of the control points are collinear. Then the rational cubic B\'ezier curve degenerates to a conic if and only if $\phi_1=0$ and $\phi_2=0.$
\end{thm}
\begin{proof}
Assume that the rational cubic curve degenerates to a conic section. Then by Theorem 1 in 
\cite{wang_1992} we have
\[
\frac{u_0\lambda_1}{u_1\lambda_0} = \frac{u_1\lambda_2}{u_2\lambda_1} = \frac{u_2\lambda_3}{u_3\lambda_2}.
\]
By a simple rearrangement of these equations it is clear that we have $\phi_1=0$ and $\phi_2=0.$ 

Assume now that $\phi_1=0$ and $\phi_2=0.$ Then, by Proposition 4, we have $\phi_3=0$ and therefore $r(t)\equiv 0$ by definition. 
Therefore, by Lemma \ref{lem:degen}, we have $q_2(\mathbf{p}(t))\equiv 0,$ which shows that $q_2$ is the implicit representation of $\mathbf{p}.$ Since $q_2$ is 
a quadratic function, the rational cubic curve must degenerate to a conic. 
\end{proof}

A consequence of the previous propositions, and the linear independency of the basis functions (see \ref{sec:proof}), is as follows.
\begin{cor}
Suppose that no three control points are collinear and the weights are non-zero. Then the following three statements are equivalent:
\begin{enumerate}
\item the rational cubic curve degenerates to a conic,
\item the coefficients $(b_i)_{i=0}$ are all zero,
\item the implicit polynomial $q$ vanishes identically.
\end{enumerate}
\end{cor}
\begin{proof} \ 

$(1) \implies (2):$ By Theorem \ref{thm:conic} we have $\phi_1=\phi_2=0,$ and thus $\phi_3=0$ by Proposition \ref{prop:p1p2p3}. Then, by definition
$b_i=0$ for $i=0,1,2,3.$

$(2) \implies (3):$ Trivial.

$(3) \implies (1):$ By the linear independence of the basis functions, the coefficients must all be zero. Then, since the $\lambda_i$s and $u_i$s are 
non-zero, we have $\phi_1=\phi_2=0.$ The result follows from Theorem \ref{thm:conic}.

\end{proof}

Since the implicit polynomial $q$ vanishes identically when the rational cubic curve degenerates to a conic, we need an alternative definition to represent this case implicitly. Another immediate consequence of Theorem \ref{thm:conic} gives such a representation.
\begin{cor}\label{cor:conicequation}
If the rational cubic curve $\mathbf{p}(t)$ degenerates to a conic section then its implicit representation is given by the equation 
\[
q_2(x,y)=u_0 u_3 L_{03}(x,y)^2 - u_1 u_2 L_{01}(x,y)L_{23}(x,y)=0.
\]
\end{cor}
\begin{proof}
The proof is immediate from Lemma \ref{lem:degen} and Theorem \ref{thm:conic}.
\end{proof}

\subsection{Class conditions for rational cubic B\'ezier curve that degenerate to conics}

Class conditions for when the conic section represents an elliptic, a parabolic or a hyperbolic segment, are described by Wang and Wang in Theorem 3 of \cite{wang_1992}. 
For completion, we state the results of that theorem here. Define two quantities $Y_1$ and $Y_2$ such that 
\[
Y_1 = \frac{\Vert \mathbf{c}_1-\mathbf{c}_* \Vert}{\Vert \mathbf{c}_0-\mathbf{c}_1 \Vert}, \ 
Y_2 = \frac{\Vert \mathbf{c}_2-\mathbf{c}_* \Vert}{\Vert \mathbf{c}_3-\mathbf{c}_2 \Vert},
\]
where $\mathbf{c}_*$ is the point of intersection of the lines $L_{01}$ and $L_{32}.$ If $L_{01}$ and $L_{32}$ are parallel, we take $Y_1=Y_2=\infty.$
From these quantities we define a number $\eta^2$ such that
\[
\eta^2 = \frac{1}{4Y_1Y_2}.
\]
The theorem then says that the curve is an ellipse when $\eta^2<1,$ a parabola when $\eta^2=1$ or a hyperbola when $\eta^2>1.$ 
The proof of that result requires only that no three control points are collinear, and is hence general enough to encompass the case of curves with non-convex polygons and negative weights.


\section{Collinear points, zero weights and numerical stability}\label{sec:collinear}

\subsection{Collinear and coincident points}

As emphasised earlier, the method fails when any three of the four control points are collinear, including the case when two of the points are coincident. 
In the collinear case, the basis functions $(K_{i})_{i=0}^3$ are no longer linearly independent, and thus do not provide a basis for the curve. One remedy for this is to
subdivide the curve using the de Casteljau algorithm, and treat the two subdivided curves separately. In the case of coincident points, two subdivisions may be necessary to completely remove the collinearity. 

So far, all the methods in this paper have been independent of the parametric form, and it would be nice to find a method of dealing with these cases without 
resorting to parametric subdivision. Unfortunately, explicit methods which incorporate all collinear configurations seem to be rather more complicated than the 
simple methods presented in this paper. It appears that different configurations of points require different basis functions and coefficients. For example, end 
point coincidence ($\mathbf{c}_0=\mathbf{c}_3$), mid point coincidence ($\mathbf{c}_1=\mathbf{c}_2$) and mixed point coincidence ($\mathbf{c}_0=\mathbf{c}_1,$ or 
$\mathbf{c}_2=\mathbf{c}_3,$ or $\mathbf{c}_0=\mathbf{c}_2,$ or $\mathbf{c}_1=\mathbf{c}_3$) all appear to require separate treatment. In addition, there are 
also several cases of collinearity to consider. 

Experimentally, it appears that the following basis functions support all cases, including collinear and coincident control points. 
\[
L_{03}^3, \quad L_{01}L_{13}^2, \quad L_{02}^2L_{23}, \quad L_{01}L_{12}L_{23},  \quad L_{12}L_{03}^2, \quad L_{02}L_{23}, \quad L_{01}L_{13}.
\]
However, the explicit formula for the coefficients no longer holds, and the symmetries that were apparent earlier in this paper, appear to be lost. Due to the 
number of different cases and the comparative complexity of an explicit formula when trying to incorporate collinear points, we feel that parametric subdivision 
of the curve is currently the best option. However, this is the subject of ongoing research.

It may be noted that the test for cases of collinearity is not a difficult one. Since, during the implicitization we are using the $\lambda_i$s as coefficients, 
if one is computed to be zero, we can instruct the algorithm to deal with that case appropriately. 
In the case of running the algorithm in floating point precision, we would specify collinearity to within given tolerance.

\subsection{Zero weights}
 
In the preceding sections we have mostly assumed that the weights are non-zero. In the CAGD community, it is fairly common to 
define rational B\'ezier curves as having non-zero, or positive weights. Indeed, if we allow some of the weights of a rational cubic curve to be zero, the curve often 
degenerates to a conic or a line. In such cases it would be better to model the curve as a lower degree parametric curve. However, if either $w_1=0$ or $w_2=0$ 
(but not both), then the curve does not degenerate to a conic section. These cases were treated in Proposition \ref{prop:endpoints}. The implicit representation 
given by Theorem \ref{thm:main}, is still valid in these cases.

\subsection{Numerical issues} 

When the curve has control points which are close to collinear, or the curve is close to a degenerate conic, issues with numerical stability need to be considered. 
Heuristically, it seems that issues with numerical stability are not too great; the methods appear to work well even when these `close to degenerate' cases 
occur. However, when implementing in a given floating point precision, the tolerances required to define when a case is considered degenerate should be investigated 
further.

\section{Examples}\label{sec:examples}

In this section we consider several examples for which the computations can easily be done by hand. Figure \ref{fig:examples} shows five different curves with various properties. We also include cases which fail using the general method and need to be treated separately. 

\subsection{Three simple examples}

For simplicity, the control points of the first three examples of Figure \ref{fig:examples} are rearrangements of the points $(0,0)^T,$ $(0,1)^T,$ $(1,1)^T$ and $(1,0)^T.$ 
For each case we compute the quantities $(\lambda_i)_{i=0}^3$ and $(\phi_i)_{i=1}^3,$ the coefficients $(b_i)_{i=0}^3,$ 
the double point $\mathbf{s},$ and whether or not the curve exhibits an unwanted singularity. In this section we describe the first example in detail.

We have 
\begin{equation*}
\mathbf{c}_0 = \colvec{0 \\ 0}, \ \mathbf{c}_1 = \colvec{0 \\ 1}, \ \mathbf{c}_2 = \colvec{1 \\ 1}, \ \mathbf{c}_3 = \colvec{1 \\ 0}, \\
\end{equation*}
and
\begin{equation*}
w_0 = w_1 = w_2 = w_3 = 1.
\end{equation*}
The quantities $\lambda_i$ are thus all equal to $\pm1$ and we have $u_0=u_3=1$ and $u_1=u_2=3.$ Using the formula (\ref{eq:phicoefs}), we obtain
\[
\phi_1 = -6, \ \phi_2 = -6 \text{ and } \phi_3 = -8 
\]
and thus
\[
b_0 = 72, \ b_1 = -18, \ b_2 = -18, \ b_3 = 8.
\]
The double point of the curve can be computed from (\ref{eq:sreduced1}) as 
\begin{equation*}
\mathbf{s} = \frac{-324(1,1)^T+432(1,0)^T}{216} = \left(\frac{1}{2},-\frac{3}{2}\right).
\end{equation*}
Clearly the double point does not lie within the convex hull of the control points.  
For the implicit equation in a Cartesian system we can write
\[
q(x,y) = 72 x (y-1) (x-1) - 18 x (1-x-y)^2 - 18 (x-y)^2 (x-1) + 8 y^3.
\]
Since $\phi_1\phi_2>0,$ we know, by Proposition \ref{prop:unwanted}, that there is no unwanted branch in the region of interest.

\begin{table}
\centering
\begin{tabular}{|l||c|c|c|c|c|c|c|}
\hline 
       & $(\lambda_0,\lambda_1,\lambda_2,\lambda_3)$ & $(b_0,b_1,b_2,b_3)$  & $(\phi_1,\phi_2,\phi_3)$ & $\mathbf{s}$                 & $\phi_1\phi_2$ \\ \hline\hline
Ex. 1  & $(1,-1,1,-1)$                               &     $(72,-18,-18,8)$ &             $(-6,-6,-8)$ & $(\frac{1}{2},\frac{-3}{2})$ &     36         \\ \hline
Ex. 2  & $(-1,-1,1,1)$                               &     $(72,-36,-36,8)$ &             $(12,12,-8)$ & $(\frac{1}{2},\frac{3}{4})$  &     144        \\ \hline
Ex. 3  & $(-1,1,1,-1)$                               &     $(72,-36,-36,8)$ &              $(12,12,8)$ & $(\infty,\infty)$            &     144        \\ \hline
Ex. 4  & $(1/3,-1,1,-1/3)$                           &       $(0, 0, 0, 0)$ &                $(0,0,0)$ & n/a                          &     0          \\ \hline
Ex. 5  & $(-1/2,0,1,-1/2)$                           & $(0, 0, -9/2, 9/16)$ &            $(9/2,3,9/4)$ & n/a                          &     n/a        \\ \hline
\end{tabular}
\caption{The computed quantities for a range of curves pictured in Figure \ref{fig:examples}. All examples use the same weights $u_0=u_3=1$ and $u_1=u_2=3.$}
\label{tab:examples}
\end{table}

We summarize the quantities for the five examples of Figure \ref{fig:examples}, in Table \ref{tab:examples}. Note that in the third example we have a double point at 
infinity. In this case, the denominator in the formula (\ref{eq:sreduced1}) vanishes, as expected. 

\begin{figure}
\begin{center}
\subfloat{
\begin{tikzpicture}[scale=1.3]
\path (-0.7,0.5) coordinate (P);
\path (0,0) coordinate (A);
\path (0,0.05) coordinate (Ap);
\path (0,-0.05) coordinate (Am);
\path (0,1) coordinate (B);
\path (0,1+0.05) coordinate (Bp);
\path (0,1-0.05) coordinate (Bm);
\path (1,1) coordinate (C);
\path (1,1+0.05) coordinate (Cp);
\path (1,1-0.05) coordinate (Cm);
\path (1,0) coordinate (D);
\path (1,+0.05) coordinate (Dp);
\path (1,-0.05) coordinate (Dm);
\draw (A) -- (B);
\draw (B) -- (C);
\draw (C) -- (D);
\draw[fill] (A) circle (0.05);
\draw[fill] (B) circle (0.05);
\draw[fill] (C) circle (0.05);
\draw[fill] (D) circle (0.05);
\draw (-0.4,0.5) node[right] {$b_0$};
\end{tikzpicture} 
\begin{tikzpicture}[scale=1.3]
\draw (A) -- (B);
\draw (Bp) -- (Dp);
\draw (Bm) -- (Dm);
\draw[dotted] (B) -- (C);
\draw[dotted] (C) -- (D);
\draw[fill] (A) circle (0.05);
\draw[fill] (B) circle (0.05);
\draw[fill] (C) circle (0.05);
\draw[fill] (D) circle (0.05);
\draw (P) node[right] {$+$ $b_1$};
\end{tikzpicture}
\begin{tikzpicture}[scale=1.3]
\draw (Ap) -- (Cp);
\draw (Am) -- (Cm);
\draw (C) -- (D);
\draw[dotted] (A) -- (B);
\draw[dotted] (B) -- (C);
\draw[fill] (A) circle (0.05);
\draw[fill] (B) circle (0.05);
\draw[fill] (C) circle (0.05);
\draw[fill] (D) circle (0.05);
\draw (P) node[right] {$+$ $b_2$};
\end{tikzpicture} 
\begin{tikzpicture}[scale=1.3]
\draw (A) -- (D);
\draw (Ap) -- (Dp);
\draw (Am) -- (Dm);
\draw[dotted] (A) -- (B);
\draw[dotted] (B) -- (C);
\draw[dotted] (C) -- (D);
\draw[fill] (A) circle (0.05);
\draw[fill] (B) circle (0.05);
\draw[fill] (C) circle (0.05);
\draw[fill] (D) circle (0.05);
\draw (P) node[right] {$+$ $b_3$};
\end{tikzpicture} 
\begin{tikzpicture}[scale=1.3]
\draw (A) .. controls (B) and (C) .. (D);
\draw[dotted] (A) -- (B);
\draw[dotted] (B) -- (C);
\draw[dotted] (C) -- (D);
\draw[fill] (A) circle (0.05);
\draw[fill] (B) circle (0.05);
\draw[fill] (C) circle (0.05);
\draw[fill] (D) circle (0.05);
\draw (P) node[right] {$=$};
\end{tikzpicture}}\caption*{Example 1: Non-singular segment}  
\subfloat{
\begin{tikzpicture}[scale=1.3]
\path (-0.7,0.5) coordinate (P);
\path (0,0) coordinate (A);
\path (0,0.05) coordinate (Ap);
\path (0,-0.05) coordinate (Am);
\path (-0.03,0) coordinate (Al);
\path (0.03,0) coordinate (Ar);
\path (0,1) coordinate (C);
\path (0.03,1) coordinate (Cp);
\path (-0.03,1) coordinate (Cm);
\path (1,1) coordinate (B);
\path (1+0.03,1) coordinate (Bp);
\path (1-0.03,1) coordinate (Bm);
\path (1,0) coordinate (D);
\path (1,+0.05) coordinate (Dp);
\path (1,-0.05) coordinate (Dm);
\path (1.03,0) coordinate (Dr);
\path (0.97,0) coordinate (Dl);
\draw (A) -- (B);
\draw (B) -- (C);
\draw (C) -- (D);
\draw[fill] (A) circle (0.05);
\draw[fill] (B) circle (0.05);
\draw[fill] (C) circle (0.05);
\draw[fill] (D) circle (0.05);
\draw (-0.4,0.5) node[right] {$b_0$};
\end{tikzpicture} 
\begin{tikzpicture}[scale=1.3]
\draw (A) -- (B);
\draw (Bp) -- (Dr);
\draw (Bm) -- (Dl);
\draw[dotted] (B) -- (C);
\draw[dotted] (C) -- (D);
\draw[fill] (A) circle (0.05);
\draw[fill] (B) circle (0.05);
\draw[fill] (C) circle (0.05);
\draw[fill] (D) circle (0.05);+12
\draw (P) node[right] {$+$ $b_1$};
\end{tikzpicture}
\begin{tikzpicture}[scale=1.3]
\draw (Ar) -- (Cp);
\draw (Al) -- (Cm);
\draw (C) -- (D);
\draw[dotted] (A) -- (B);
\draw[dotted] (B) -- (C);
\draw[fill] (A) circle (0.05);
\draw[fill] (B) circle (0.05);
\draw[fill] (C) circle (0.05);
\draw[fill] (D) circle (0.05);
\draw (P) node[right] {$+$ $b_2$};
\end{tikzpicture} 
\begin{tikzpicture}[scale=1.3]
\draw (A) -- (D);
\draw (Ap) -- (Dp);
\draw (Am) -- (Dm);
\draw[dotted] (A) -- (B);
\draw[dotted] (B) -- (C);
\draw[dotted] (C) -- (D);
\draw[fill] (A) circle (0.05);
\draw[fill] (B) circle (0.05);
\draw[fill] (C) circle (0.05);
\draw[fill] (D) circle (0.05);
\draw (P) node[right] {$+$ $b_3$};
\end{tikzpicture} 
\begin{tikzpicture}[scale=1.3]
\draw (A) .. controls (B) and (C) .. (D);
\draw[dotted] (A) -- (B);
\draw[dotted] (B) -- (C);
\draw[dotted] (C) -- (D);
\draw[fill] (A) circle (0.05);
\draw[fill] (B) circle (0.05);
\draw[fill] (C) circle (0.05);
\draw[fill] (D) circle (0.05);
\draw (P) node[right] {$=$};
\end{tikzpicture}}\caption*{Example 2: Singular segment (cusp)} 
\subfloat{
\begin{tikzpicture}[scale=1.3]
\path (-0.7,0.5) coordinate (P);
\path (0,0) coordinate (A);
\path (-0.03535,0.03535) coordinate (Ap);
\path (0.03535,-0.03535) coordinate (Am);
\path (0,0.03) coordinate (At);
\path (0,-0.03) coordinate (Ab);
\path (0,1) coordinate (B);
\path (0,1.03) coordinate (Bp);
\path (0,0.97) coordinate (Bm);
\path (1,1) coordinate (D);
\path (1-0.03535,1+0.03535) coordinate (Dp);
\path (1+0.03535,1-0.03535) coordinate (Dm);
\path (1,1+0.03) coordinate (Dt);
\path (1,1-0.03) coordinate (Db);
\path (1,0) coordinate (C);
\path (1,+0.05) coordinate (Cp);
\path (1,-0.05) coordinate (Cm);
\path (1,0.03) coordinate (Ct);
\path (1,-0.03) coordinate (Cb);
\draw (A) -- (B);
\draw (B) -- (C);
\draw (C) -- (D);
\draw[fill] (A) circle (0.05);
\draw[fill] (B) circle (0.05);
\draw[fill] (C) circle (0.05);
\draw[fill] (D) circle (0.05);
\draw (-0.4,0.5) node[right] {$b_0$};
\end{tikzpicture} 
\begin{tikzpicture}[scale=1.3]
\draw (A) -- (B);
\draw (Bp) -- (Dt);
\draw (Bm) -- (Db);
\draw[dotted] (C) -- (D);
\draw[fill] (A) circle (0.05);
\draw[fill] (B) circle (0.05);
\draw[fill] (C) circle (0.05);
\draw[fill] (D) circle (0.05);
\draw (P) node[right] {$+$ $b_1$};
\end{tikzpicture}
\begin{tikzpicture}[scale=1.3]
\draw (At) -- (Ct);
\draw (Ab) -- (Cb);
\draw (C) -- (D);
\draw[dotted] (A) -- (B);
\draw[dotted] (B) -- (C);
\draw[fill] (A) circle (0.05);
\draw[fill] (B) circle (0.05);
\draw[fill] (C) circle (0.05);
\draw[fill] (D) circle (0.05);
\draw (P) node[right] {$+$ $b_2$};
\end{tikzpicture} 
\begin{tikzpicture}[scale=1.3]
\draw (A) -- (D);
\draw (Ap) -- (Dp);
\draw (Am) -- (Dm);
\draw[dotted] (A) -- (B);
\draw[dotted] (B) -- (C);
\draw[dotted] (C) -- (D);
\draw[fill] (A) circle (0.05);
\draw[fill] (B) circle (0.05);
\draw[fill] (C) circle (0.05);
\draw[fill] (D) circle (0.05);
\draw (P) node[right] {$+$ $b_3$};
\end{tikzpicture} 
\begin{tikzpicture}[scale=1.3]
\draw (A) .. controls (B) and (C) .. (D);
\draw[dotted] (A) -- (B);
\draw[dotted] (B) -- (C);
\draw[dotted] (C) -- (D);
\draw[fill] (A) circle (0.05);
\draw[fill] (B) circle (0.05);
\draw[fill] (C) circle (0.05);
\draw[fill] (D) circle (0.05);
\draw (P) node[right] {$=$};
\end{tikzpicture}}\caption*{Example 3: Double point at infinity}
\subfloat{
\begin{tikzpicture}[scale=1.3]
\path (-0.7,0.5) coordinate (P);
\path (0,0) coordinate (A);
\path (0,0.05) coordinate (Ap);
\path (0,-0.05) coordinate (Am);
\path (1/3,1) coordinate (B);
\path (1/3,1+0.05) coordinate (Bp);
\path (1/3,1-0.05) coordinate (Bm);
\path (2/3,1) coordinate (C);
\path (2/3,1+0.05) coordinate (Cp);
\path (2/3,1-0.05) coordinate (Cm);
\path (1,0) coordinate (D);
\path (1,+0.05) coordinate (Dp);
\path (1,-0.05) coordinate (Dm);
\draw (A) -- (B);
\draw (B) -- (C);
\draw (C) -- (D);
\draw[fill] (A) circle (0.05);
\draw[fill] (B) circle (0.05);
\draw[fill] (C) circle (0.05);
\draw[fill] (D) circle (0.05);
\draw (-0.4,0.5) node[right] {$b_0$};
\end{tikzpicture} 
\begin{tikzpicture}[scale=1.3]
\draw (A) -- (B);
\draw (Bp) -- (Dp);
\draw (Bm) -- (Dm);
\draw[dotted] (B) -- (C);
\draw[dotted] (C) -- (D);
\draw[fill] (A) circle (0.05);
\draw[fill] (B) circle (0.05);
\draw[fill] (C) circle (0.05);
\draw[fill] (D) circle (0.05);
\draw (P) node[right] {$+$ $b_1$};
\end{tikzpicture}
\begin{tikzpicture}[scale=1.3]
\draw (Ap) -- (Cp);
\draw (Am) -- (Cm);
\draw (C) -- (D);
\draw[dotted] (A) -- (B);
\draw[dotted] (B) -- (C);
\draw[fill] (A) circle (0.05);
\draw[fill] (B) circle (0.05);
\draw[fill] (C) circle (0.05);
\draw[fill] (D) circle (0.05);
\draw (P) node[right] {$+$ $b_2$};
\end{tikzpicture} 
\begin{tikzpicture}[scale=1.3]
\draw (A) -- (D);
\draw (Ap) -- (Dp);
\draw (Am) -- (Dm);
\draw[dotted] (A) -- (B);
\draw[dotted] (B) -- (C);
\draw[dotted] (C) -- (D);
\draw[fill] (A) circle (0.05);
\draw[fill] (B) circle (0.05);
\draw[fill] (C) circle (0.05);
\draw[fill] (D) circle (0.05);
\draw (P) node[right] {$+$ $b_3$};
\end{tikzpicture} 
\begin{tikzpicture}[scale=1.3]
\draw (A) .. controls (B) and (C) .. (D);
\draw[dotted] (A) -- (B);
\draw[dotted] (B) -- (C);
\draw[dotted] (C) -- (D);
\draw[fill] (A) circle (0.05);
\draw[fill] (B) circle (0.05);
\draw[fill] (C) circle (0.05);
\draw[fill] (D) circle (0.05);
\draw (P) node[right] {$\neq$};
\end{tikzpicture}}\caption*{Example 4: Degeneration to a conic}
\subfloat{
\begin{tikzpicture}[scale=1.3]
\path (-0.7,0.5) coordinate (P);
\path (0,0) coordinate (A);
\path (0,0.05) coordinate (Ap);
\path (0,-0.05) coordinate (Am);
\path (0,1) coordinate (B);
\path (0,1+0.05) coordinate (Bp);
\path (0,1-0.05) coordinate (Bm);
\path (0.5,0.0) coordinate (C);
\path (0.5,0.0+0.05) coordinate (Cp);
\path (0.5,0.0-0.05) coordinate (Cm);
\path (1,0.0) coordinate (D);
\path (1,0.0+0.05) coordinate (Dp);
\path (1,0.0-0.05) coordinate (Dm);
\draw (A) -- (B);
\draw (B) -- (C);
\draw (A) -- (D);
\draw[fill] (A) circle (0.05);
\draw[fill] (B) circle (0.05);
\draw[fill] (C) circle (0.05);
\draw[fill] (D) circle (0.05);
\draw (-0.4,0.5) node[right] {$b_0$};
\end{tikzpicture} 
\begin{tikzpicture}[scale=1.3]
\draw (A) -- (B);
\draw (Bp) -- (Dp);
\draw (Bm) -- (Dm);
\draw[dotted] (B) -- (C);
\draw[dotted] (C) -- (D);
\draw[fill] (A) circle (0.05);
\draw[fill] (B) circle (0.05);
\draw[fill] (C) circle (0.05);
\draw[fill] (D) circle (0.05);
\draw (P) node[right] {$+$ $b_1$};
\end{tikzpicture} 
\begin{tikzpicture}[scale=1.3]
\draw (Ap) -- (Cp);
\draw (A) -- (C);
\draw (C) -- (D);
\draw (A) -- (D);
\draw (Ap) -- (Dp);
\draw (Am) -- (Dm);
\draw[dotted] (A) -- (B);
\draw[dotted] (B) -- (C);
\draw[fill] (A) circle (0.05);
\draw[fill] (B) circle (0.05);
\draw[fill] (C) circle (0.05);
\draw[fill] (D) circle (0.05);
\draw (P) node[right] {$+$ $b_2$};
\end{tikzpicture}  
\begin{tikzpicture}[scale=1.3]
\draw (A) -- (D);
\draw (Ap) -- (Dp);
\draw (Am) -- (Dm);
\draw[dotted] (A) -- (B);
\draw[dotted] (B) -- (C);
\draw[dotted] (C) -- (D);
\draw[fill] (A) circle (0.05);
\draw[fill] (B) circle (0.05);
\draw[fill] (C) circle (0.05);
\draw[fill] (D) circle (0.05);
\draw (P) node[right] {$+$ $b_3$};
\end{tikzpicture} 
\begin{tikzpicture}[scale=1.3]
\draw (A) .. controls (B) and (C) .. (D);
\draw[dotted] (A) -- (B);
\draw[dotted] (B) -- (C);
\draw[dotted] (C) -- (D);
\draw[fill] (A) circle (0.05);
\draw[fill] (B) circle (0.05);
\draw[fill] (C) circle (0.05);
\draw[fill] (D) circle (0.05);
\draw (P) node[right] {$\neq$};
\end{tikzpicture}}\caption*{Example 5: Collinear control points}
\end{center}
\caption{The quantities computed for each of these examples are shown in Table \ref{tab:examples}. The fourth and fifth examples fail. The fourth is a curve which 
degenerates to a conic, whereas the fifth has three collinear points. These cases need to be treated separately.}\label{fig:examples}
\end{figure}

\subsection{Example of regular conic degeneration}

This is the fourth example of Figure \ref{fig:examples}. Having detected that both $\phi_1=0$ and $\phi_2=0,$ Theorem \ref{thm:conic} tells us we have a conic 
section. We thus use the formula of Corollary \ref{cor:conicequation} for the implicit representation. This gives,
\begin{equation*}
\begin{split}
q(x,y) & = u_0 u_3 L_{03}(x,y)^2-u_1 u_2 L_{01}(x,y) L_{23}(x,y), \\
       & = y^2 + 9(x - y/3)(x + y/3 - 1).
\end{split}
\end{equation*}
Since we have a conic section the double point computations are not applicable. 

Using the conic class conditions for this example we compute $Y_1=Y_2=1/2,$ which gives $\eta^2=1.$ 
This shows that we have a parabola, which is consistent with the curve pictured in Example 4 of Figure \ref{fig:examples}.

\subsection{Example of conic degeneration with negative weights}

As mentioned previously, the result of Theorem \ref{thm:conic} is not restricted to convex polygons with non-negative weights. 
Consider, for example, the curve pictured in Figure \ref{fig:conic_neg}, which is given by the control points 
\[
\mathbf{c}_0 = \colvec{1/2 \\ 0}, \ \mathbf{c}_1 = \colvec{0 \\ 1}, \ \mathbf{c}_2 = \colvec{1 \\ 1/2}, \ \mathbf{c}_3 = \colvec{0 \\ 0}.
\]
Using the results of Theorem \ref{thm:conic}, we can find a real conic section corresponding to these points. 
Without loss of generality, we can assume $u_0=u_3=1$ \cite{farin_2002}. 
Using this assumption, we can solve the equations of Theorem \ref{thm:conic} for the remaining weights. We thus obtain 
\[
u_1 = \frac{\lambda_1}{\left(\lambda_0^2\lambda_3\right)^{1/3}} = -(1/48)^{1/3}, \quad
u_2 = \frac{\lambda_2}{\left(\lambda_0\lambda_3^2\right)^{1/3}} = -(2/9)^{1/3}.
\]
The weights corresponding to the real cube roots of these quantities determine the conic pictured in Figure \ref{fig:conic_neg}. 
By construction, we have $\phi_1=\phi_2=0.$ The implicit equation is given by 
\[
q(x,y) = \frac{1}{24}\left((-2x-y+1)(x-2y) + 6y^2\right).
\]
It should be noted that this construction can be used to determine a real conic section corresponding to \emph{any} set of control points such that no three are collinear.
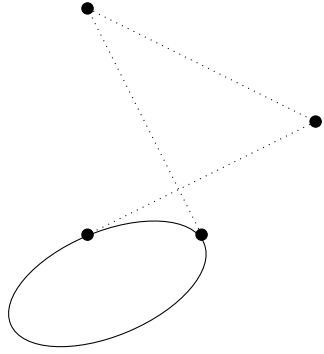
\begin{figure}
\begin{center}
\begin{tikzpicture}[scale=3]
  \path (1/2, 0) coordinate (A0);
  \path (0, 1)coordinate (B0);
  \path (1, 1/2) coordinate (C0);
  \path (0, 0) coordinate (D0);
  \draw[dotted] (A0) -- (B0);
  \draw[dotted] (B0) -- (C0);
  \draw[dotted] (C0) -- (D0);
  \draw[rotate around={23:(2/23,-5/23)}] (2/23,-5/23) ellipse (0.46 and 0.23);
  \draw[fill] (A0) circle (0.025);
  \draw[fill] (B0) circle (0.025);
  \draw[fill] (C0) circle (0.025);
  \draw[fill] (D0) circle (0.025);
\end{tikzpicture}
\end{center}
\caption{An example of conic degeneration in a curve with non-convex control polygon.}\label{fig:conic_neg}
\end{figure}
The proof of the conic class conditions in \cite{wang_1992} is also fully general to control polygons with no three collinear control points. In our example, we have 
\[
Y_1 = 4/5, \ Y_2 = 3/5, \ \eta^2 = 25/48 < 1,
\]
thus confirming that we have an ellipse.

\subsection{Example with collinear control points}

The disappearance of $\lambda_1$ in the example pictured in Figure \ref{fig:example_collinear} indicates that a collinearity occurs between $\mathbf{c}_0,$ $\mathbf{c}_2$ and $\mathbf{c}_3.$ For this example, the weights $w_i$ are all assumed to be equal to $1.$ We see that there 
appears a linear dependency in the basis functions, between $K_2$ and $K_3,$ and the coefficients $b_0$ and $b_1$ become zero. Following the suggestion of Section 
\ref{sec:collinear}, we therefore subdivide the curve a single time, at the parameter value $t=1/2.$ This gives two curves with the following control points 
\[
\mathbf{c}_{1,0} = \colvec{0 \\ 0}, \mathbf{c}_{1,1} = \colvec{0 \\ 1/2}, \mathbf{c}_{1,2} = \colvec{1/8 \\ 1/2}, \mathbf{c}_{1,3} = \colvec{5/16 \\ 3/8},
\]
and 
\[
\mathbf{c}_{2,0} = \colvec{5/16 \\ 3/8}, \mathbf{c}_{2,1} = \colvec{1/2 \\ 1/4}, \mathbf{c}_{2,2} = \colvec{3/4 \\ 0}, \mathbf{c}_{2,3} = \colvec{1 \\ 0}.
\]
These can each be treated in the same way as the previous examples, finding in both cases that $\mathbf{s}=(-8,36)^T;$ an acnode. 

It may be noted that although the implicit equation vanishes identically, the double point can still be computed without subdivision if we choose the correct formula; 
that is by choosing the barycentric formula (\ref{eq:sreduced2}), with respect to the three non-collinear points $\mathbf{c}_0,$ $\mathbf{c}_1 $ and $\mathbf{c}_3.$ 
We then get 
\[
\mathbf{s} = \frac{-243/2(0,1)^T+27(1,0)}{-27/8} = \colvec{-8 \\ 36}.
\]
\begin{figure}
\begin{center}
\begin{tikzpicture}[scale=3]
\path (0,0) coordinate (A);
\path (0,0.5) coordinate (B);
\path (0.125,0.5) coordinate (C);
\path (5/16,3/8) coordinate (D);
\path (0.5,0.25) coordinate (E);
\path (0.75,0.0) coordinate (F);
\path (1,0) coordinate (G);
\path (0,1) coordinate (BC);
\path (0.5,0) coordinate (EF);
\draw[dotted] (A) -- (BC);
\draw[dotted] (BC) -- (EF);
\draw[dotted] (EF) -- (G);
\draw[fill] (A) circle (0.025);
\draw[fill] (BC) circle (0.025);
\draw[fill] (EF) circle (0.025);
\draw[fill] (G) circle (0.025);
\draw (A) .. controls (B) and (C) .. (D);
\draw (D) .. controls (E) and (F) .. (G);
\draw (0.5,0.5) node[right] {$\mathbf{p}(t)$};
\end{tikzpicture} \hspace*{1cm}
\begin{tikzpicture}[scale=3]
\draw[dotted] (A) -- (B);
\draw[dotted] (B) -- (C);
\draw[dotted] (C) -- (D);
\draw[fill] (A) circle (0.025);
\draw[fill] (B) circle (0.025);
\draw[fill] (C) circle (0.025);
\draw[fill] (D) circle (0.025);
\draw (A) .. controls (B) and (C) .. (D);
\draw (0.1,0.2) node[right] {$\mathbf{p}_1(t)$};
\end{tikzpicture} \hspace*{1cm}
\begin{tikzpicture}[scale=3]
\path (5/16,3/8) coordinate (D2);
\path (0.5,0.25) coordinate (E2);
\path (0.75,0.0) coordinate (F2);
\path (1,0) coordinate (G2);
\draw[dotted] (D2) -- (E2);
\draw[dotted] (E2) -- (F2);
\draw[dotted] (F2) -- (G2);
\draw[fill] (D2) circle (0.025);
\draw[fill] (E2) circle (0.025);
\draw[fill] (F2) circle (0.025);
\draw[fill] (G2) circle (0.025);
\draw (D2) .. controls (E2) and (F2) .. (G2);
\draw (0.8,0.3) node[right] {$\mathbf{p}_2(t)$};
\end{tikzpicture} 
\end{center}
\caption{An example with collinear control points. The original curve $\mathbf{p}(t)$ (left) is subdivided into two curves, $\mathbf{p}_1(t)$ (centre) and $\mathbf{p}_2(t)$ (right), using the de Casteljau algorithm, either of which can be used for the implicitization.}\label{fig:example_collinear}
\end{figure}
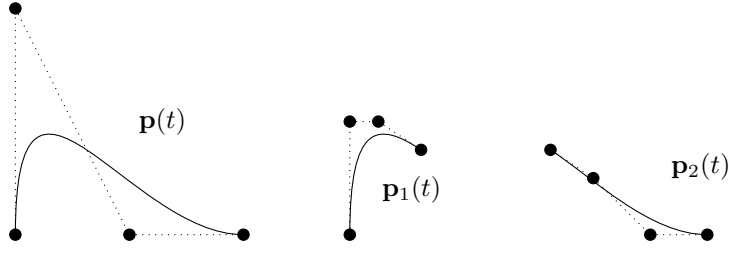

\subsection{Example with an unwanted singularity}

For this example we use the control points for the curve pictured in Figure \ref{fig:singlines}, given by
\[
\mathbf{c}_0 = \colvec{1/4 \\ 0}, \ \mathbf{c}_1 = \colvec{9/8 \\ 1/2}, \ \mathbf{c}_2 = \colvec{13/16 \\ 3/4}, \ \mathbf{c}_3 = \colvec{17/32 \\ 19/24}, 
\]
with $w_0=w_1=w_2=w_3=1.$ Using the same computations as in the previous examples we get 
\begin{equation*}
\begin{split}
(\lambda_i)_{i=0}^3 & = (11/192, 15/64, 53/96, 3/8), \\
(b_i)_{i=0}^3       & = (312435/4194304, -66285/2097152, 220957/18874368, 1441/1048576), \\
(\phi_i)_{i=1}^3    & =  (-491/4096, 379/3072, 131/2048), \\ 
\mathbf{s}          & = (363241/470596, 146294/352947), \\
\phi_1\phi_2        & = -186089/12582912.
\end{split}
\end{equation*}
The negative value of $\phi_1\phi_2$ indicates that we have an unwanted singularity, the location of which is given by $\mathbf{s}.$ Since the curve
exhibits an unwanted singularity, we compute the lines $\tilde{S}_1$ and $\tilde{S}_2,$ which intersect the double point:
\begin{equation*}
\begin{split}
\tilde{S}_1(x,y) & = 965/8192x - 1215/8192y - 965/32768, \\
\tilde{S}_2(x,y) & = -12773/49152x - 10865/65536y + 17649/65536.
\end{split}
\end{equation*}
For a given point $(x,y),$ we can then use boolean operations on the signs of $q(x,y),$ $\tilde{S}_1(x,y)$ and $\tilde{S}_2(x,y)$ in order to define which points lie 
`inside' and `outside' the curve.

\section{Discussion and conclusion}\label{sec:conc}

For the sake of brevity, we have omitted extended discussions of interesting features that arise in using this method in the previous sections. 
In this section we mention some of these features before concluding the paper. 

An alternative representation for the implicit coefficients is to divide through by the non-zero factor 
$u_0 u_1 u_2 u_3 \lambda_0 \lambda_1 \lambda_2 \lambda_3:$
\begin{eqnarray*}
\tilde{b}_0 & = &  \frac{u_1 u_2}{u_0 u_3}-\frac{\lambda_1 \lambda_2}{\lambda_0 \lambda_3}, \\
\tilde{b}_1 & = & \frac{\lambda_1^2}{\lambda_0 \lambda_2} - \frac{u_1^2}{u_0 u_2},   \\
\tilde{b}_2 & = & \frac{\lambda_2^2}{\lambda_1 \lambda_3} - \frac{u_2^2}{u_1 u_3},    \\
\tilde{b}_3 & = & \frac{\lambda_0 \lambda_3}{\lambda_1 \lambda_2} - \frac{u_0 u_3}{u_1 u_2}.
\end{eqnarray*}
This form highlights more clearly the symmetry between the $u_i$s and the $\lambda_i$s. It also appears that this approach aids the numerical stability 
of the implementation, being only quadratic in the numerators and denominators.

It is interesting to note that the exponents which appear in the formulas for the coefficients (\ref{eq:coefs}), are reminiscent of the exponents 
of the terms of the discriminant of a univariate cubic polynomial in monomial form\footnote{There is a fifth term in the cubic discriminant that 
would correspond to $b_4 = u_0 u_1 u_2 u_3 \lambda_0 \lambda_1 \lambda_2 \lambda_3 - u_0 u_1 u_2 u_3 \lambda_0 \lambda_1 \lambda_2 \lambda_3 \equiv 0.$}. 
The relationship between the coefficients and the cubic discriminant should be the subject of further research.

Experiments show that parts of the method appear to be extensible to higher degrees. In particular, it is not difficult to define basis functions
for rational quartic curves using the same heuristic reasoning as in Section \ref{sec:thebasis}. However, the number of basis functions appears to increase
exponentially, and attempts to find an explicit formula for the coefficients appear to be more difficult. An extension of the theory to surfaces also 
appears to be much more difficult due to the complicated limiting control surfaces involved. However, this could be a direction for future research.

\subsection{Conclusion}

This paper has shown that it is possible to represent the implicit form of all non-degenerate rational planar cubic B\'ezier 
curves as a linear combination of four basis functions, whenever no three control points are collinear. The method has been described in
terms of purely geometric quantities and symmetries have been highlighted. The resulting coefficients of the implicit polynomial lead 
naturally to a geometric characterization of several aspects of the curve. The method has a compact representation and can easily be 
implemented on a GPU, as an alternative to the methods in \cite{blinn_2005,pfeifle_2012}. Additionally, the formulas which aid in locating 
the singularity, and whether or not it is unwanted, are simple and computationally inexpensive.

\subsection*{Acknowledgements}

I would like to thank Tor Dokken for reading through the manuscript and for helpful suggestions. 
I would also like to thank the anonymous reviewers for their comments and corrections.
The research leading to these results has received funding from the European Community's Seventh Framework Programme FP7/2007-2013 under 
grant agreement n$^\circ$ PITN-GA-2008-214584 (SAGA), and from the Research Council of Norway (IS-TOPP).

\bibliographystyle{plain}
\nocite{*}
\bibliography{paper}

\begin{appendix}

\section{Some geometric properties}

Here we state some simple geometric properties that are used in the proofs of the next section.

\begin{prop}\label{prop:lambdapmpm}
For any four points $\mathbf{c}_0$, $\mathbf{c}_1$, $\mathbf{c}_2$ and $\mathbf{c}_3$ we have the following:
\[
\lambda_0 + \lambda_1 + \lambda_2 + \lambda_3 = 0.
\]
\end{prop}
\begin{proof}
This can be verified by simply writing out the expression using Definition \ref{def:lambda}, and checking that all terms cancel out.
\end{proof}
This proposition shows that there is some degeneracy in the representation; that is, one of the $\lambda_i$s can always be written as a combination of
the other three. This is reflected in the simplified forms presented in the paper. However, for the sake of symmetry, we have proceeded for the most 
part, to use all four $\lambda_i$ values.

\begin{prop}
Assume we are given four points $(\mathbf{c}_i)_{i=0}^3$ with no three collinear. Then, when the respective denominators are non-zero, we can define
\begin{equation}\label{eq:mpts}
\begin{split}
\mathbf{m}_1 & = \frac{\mathbf{c}_0\lambda_0+\mathbf{c}_1\lambda_1-\mathbf{c}_2\lambda_2-\mathbf{c}_3 \lambda_3}{\lambda_0+\lambda_1-\lambda_2-\lambda_3} 
               = \frac{\mathbf{c}_0\lambda_0+\mathbf{c}_1\lambda_1}{\lambda_0+\lambda_1} 
               = \frac{\mathbf{c}_3\lambda_3+\mathbf{c}_2\lambda_2}{\lambda_3+\lambda_2}, \\
\mathbf{m}_2 & = \frac{\mathbf{c}_0\lambda_0-\mathbf{c}_1\lambda_1+\mathbf{c}_2\lambda_2-\mathbf{c}_3 \lambda_3}{\lambda_0-\lambda_1+\lambda_2-\lambda_3} 
               = \frac{\mathbf{c}_0\lambda_0+\mathbf{c}_2\lambda_2}{\lambda_0+\lambda_2}
               = \frac{\mathbf{c}_1\lambda_1+\mathbf{c}_3\lambda_3}{\lambda_1+\lambda_3}, \\
\mathbf{m}_3 & = \frac{\mathbf{c}_0\lambda_0-\mathbf{c}_1\lambda_1-\mathbf{c}_2\lambda_2+\mathbf{c}_3 \lambda_3}{\lambda_0-\lambda_1-\lambda_2+\lambda_3} 
               = \frac{\mathbf{c}_0\lambda_0+\mathbf{c}_3\lambda_3}{\lambda_0+\lambda_3} 
               = \frac{\mathbf{c}_1\lambda_1+\mathbf{c}_2\lambda_2}{\lambda_1+\lambda_2}. \\
\end{split}
\end{equation}
If any of these points do exist, they define the intersection of the lines $L_{01}$ and $L_{23},$ $L_{02}$ and $L_{13},$ or 
$L_{03}$ and $L_{12}$ respectively.\footnote{At least one of these points exists in the affine plane, since if two of the denominators vanish, then the two pairs of corresponding 
lines are parallel; but then the third point will be the intersection of the two lines passing through opposite vertices of the parallelogram thus formed. These lines must 
necessarily be non-parallel, thus the point of intersection is finite.}
\end{prop}
\begin{proof}
That the various equalities hold, when the denominators are non-zero, is a consequence of Proposition \ref{prop:lambdapmpm}. The fact that they 
intersect at the respective lines is then a triviality, since each point can be written as a scaled linear combination of the points which define the line.
For example, $\mathbf{m}_1$ must lie on the line $L_{01},$ by the identity $\mathbf{m}_1 = \frac{\mathbf{c}_0\lambda_0+\mathbf{c}_1\lambda_1}{\lambda_0+\lambda_1},$
and on the line $L_{23},$ by $\mathbf{m}_1=\frac{\mathbf{c}_3 \lambda_3+\mathbf{c}_2\lambda_2}{\lambda_3+\lambda_2}.$
\end{proof}

\section{Linear independence and proofs of Theorems}\label{sec:proof}

\subsection{Linear independence}

In the following theorem we establish linear independence of the basis functions, in the applicable cases.
\begin{thm}\label{thm:linindep}
Suppose no three of the points $(\mathbf{c}_i)_{i=0}^3$ are collinear. Then the functions $K_0,K_1,K_2,K_3$ are linearly 
independent.
\end{thm}
\begin{proof}
Assume that
\[
\rho(x,y) = \sum_{i=0}^3 b_i K_i(x,y) = 0, \quad \text{for all } (x,y)\in \mathbb{R}^2.
\]
We prove linear independence by evaluating $\rho$ at four distinct points. Assume that the points $\mathbf{m}_1$ and $\mathbf{m}_2$ defined by (\ref{eq:mpts}) 
exist. Then we can evaluate $\rho$ at $\mathbf{c}_1, \mathbf{c}_2, \mathbf{m}_1$ and $\mathbf{m}_2.$ For example, at $\mathbf{m}_2$ we have 
\[
\rho(\mathbf{m}_2) = \frac{\lambda_0^2\lambda_1\lambda_2\lambda_3^2b_0 - \lambda_1^3\lambda_2^3 b_3 }{(\lambda_0+\lambda_2)^3}.
\] 
After rearranging the rows to obtain a triangular matrix and dividing through by any common factors, we can set up the linear system with respect to evaluation at 
the four points as follows:
\[
\begin{pmatrix}
\lambda_0^2\lambda_3^2   & 0                     & 0                     & -\lambda_1^2\lambda_2^2 \\ 
0                        & \lambda_0^2\lambda_3  & 0                     & -\lambda_1^3            \\
0                        & 0                     & \lambda_3^2\lambda_0  & -\lambda_2^3            \\
0                        & 0                     & 0                     & 1                         
\end{pmatrix}
\begin{pmatrix}
b_0 \\ b_1 \\ b_2 \\ b_3
\end{pmatrix}
= 0.
\]
Now, the determinant of the matrix can be computed as
\[
\lambda_0^5\lambda_3^5
\]
which never vanishes since the control points are not collinear. Care needs to be taken in the case when $L_{01}$ and $L_{23},$ (resp. $L_{02}$ and $L_{13},$) 
are parallel, as the denominator of $\mathbf{m}_1$ (resp. $\mathbf{m}_2$) vanishes. However, a similar linear system can be set up by using homogeneous 
coordinates, in which case the vanishing denominator is not a problem. Thus, the proof holds in all cases.
\end{proof}

\subsection{Proof of Theorem \ref{thm:main}}

The proof of Theorem \ref{thm:main} is essentially a long exercise in expanding the rational function $q\circ\mathbf{p},$ in order to show that it is 
identically zero. We assume the conditions of Theorem \ref{thm:main} for the entirety of this section (i.e., that the cubic is non-degenerate and no 
three control points are collinear). We know, by Theorem \ref{thm:linindep} that the basis functions $(K_i)_{i=0}^3$ are linearly independent, 
and by Proposition \ref{thm:conic} that not all the coefficients $(b_i)_{i=0}^3$ are zero. Thus the polynomial $q$ is not identically zero, and the 
theorem is proved if we can show that $q\circ\mathbf{p}$ vanishes identically.

We first consider the composition of $L_{ij}(\mathbf{p}(t))$ for all $i\neq j.$ 
\begin{lemma}\label{lem:linecomp}
The rational cubic function $L_{ij}(\mathbf{p}(t))$ can be given in Bernstein form by
\[
L_{ij}(\mathbf{p}(t)) = \frac{1}{w(t)} \sum_{k=0 \atop k \neq i,j}^3 \lambda_{ijk} u_k B_k(t),
\]
where $B_k(t) = \binom{3}{k}t^k(1-t)^{3-k}$ and $w(t)$ denotes the denominator of (\ref{eq:cubbez}).
\end{lemma}
\begin{proof}
We first note that by Definition \ref{def:L},
\[L_{ij}(\mathbf{p}(t)) = \frac{1}{w(t)} 
\begin{vmatrix}
 p_0(t) &  p_1(t) & w(t) \\
c_{i,0} & c_{i,1} & 1 \\
c_{j,0} & c_{j,1} & 1
\end{vmatrix}.
\]
Now, by expanding the determinant, we have that
\begin{equation}
\begin{split}
w(t)L_{ij}(\mathbf{p}(t)) & = (c_{i,1}-c_{j,1})\sum_{k=0}^3 c_{k,0} u_k B_k(t) + (c_{j,0}-c_{i,0})\sum_{k=0}^3 c_{k,1} u_k B_k(t) \\  
                                 & \qquad\qquad\qquad\qquad\qquad\qquad\qquad + (c_{i,0}c_{j,1}-c_{i,1}c_{j,0})\sum_{k=0}^3 u_k B_k(t),  \\
                             & = \sum_{k=0}^3 ((c_{i,1}-c_{j,1})c_{k0} + (c_{j,0}-c_{i,0})c_{k,1} + (c_{i,0}c_{j,1}-c_{i,1}c_{j,0})) u_k B_k(t), \\
                             & = \sum_{k=0}^3 \lambda_{ijk} u_k B_k(t).
\end{split}
\end{equation}
Clearly, when $k=i$ or $k=j,$ the corresponding term in the sum is zero meaning we only need sum over $k\neq i,j.$
\end{proof}

A consequence of summing only over $k\neq i,j$ is that we can remove factors of $t$ and $1-t$ when certain coefficients disappear. That is,
we can write 
\begin{equation}\label{eq:t1t}
\begin{split}
w(t)L_{01}(\mathbf{p}(t)) & = t^2 (u_2 \lambda_3(1-t) -  u_3\lambda_2 t), \\ 
w(t)L_{12}(\mathbf{p}(t)) & = (u_0 \lambda_3(1-t)^3 - u_3\lambda_0 t^3), \\ 
w(t)L_{23}(\mathbf{p}(t)) & = (1-t)^2 (u_0 \lambda_1(1-t) - u_1\lambda_0 t), \\ 
w(t)L_{02}(\mathbf{p}(t)) & = -t ( u_1\lambda_3(1-t)^2 - u_3\lambda_1 t^2), \\ 
w(t)L_{13}(\mathbf{p}(t)) & = -(1-t) ( u_0 \lambda_2(1-t)^2 - u_2\lambda_0 t^2), \\
w(t)L_{03}(\mathbf{p}(t)) & = t (1-t) (u_1\lambda_2(1-t) - u_2\lambda_1 t). 
\end{split} 
\end{equation}
Thus by Lemma \ref{lem:linecomp} and the above identities, we can express the compositions $(K_i\circ\mathbf{p})_{i=0}^3$ as follows:
\begin{lemma}\label{lem:basiscomp}
For each $i=0,1,2,3,$ we can express $K_i(\mathbf{p}(t))$ in the form 
\[
K_i(\mathbf{p}(t)) = \frac{t^2(1-t)^2}{w(t)^3}G_i(t),
\]
where
\begin{equation*}
\begin{split}
G_0(t) & = (u_2 \lambda_3(1-t) -  u_3\lambda_2 t)(u_0 \lambda_3(1-t)^3 - u_3\lambda_0 t^3)(u_0 \lambda_1(1-t) - u_1\lambda_0 t), \\
G_1(t) & = (u_2 \lambda_3(1-t) -  u_3\lambda_2 t)( u_0 \lambda_2(1-t)^2 - u_2\lambda_0 t^2)^2, \\
G_2(t) & = (u_1 \lambda_3(1-t)^2 - u_3\lambda_1 t^2)^2 (u_0 \lambda_1(1-t) - u_1\lambda_0 t), \\
G_3(t) & = t (1-t) (u_1\lambda_2(1-t) - u_2\lambda_1 t)^3. \\ 
\end{split}
\end{equation*}
\end{lemma}

The common factor of $\frac{t^2(1-t)^2}{w(t)^3}$ can be ignored in showing that $q\circ\mathbf{p}\equiv0;$ it is thus sufficient to show that 
\begin{equation}\label{eq:Gi}
\sum_{i=0}^3 b_i G_i(t) \equiv 0.
\end{equation}
It is a simple, yet lengthy exercise to compute the coefficients of this polynomial in the degree five Bernstein basis, in order to show that they are 
all zero. We compute the coefficient of $B_0^5(t)=(1-t)^5$ as an example. By observation, the coefficients $g_{i,0}$ of $B_0^5(t)$ of each of the functions 
$(G_i)_{i=0}^3$ are as follows:
\begin{equation*}
\begin{split}
g_{0,0} & = u_0^2 u_2 \lambda_1 \lambda_3^2, \\ 
g_{1,0} & = u_0^2 u_2 \lambda_2^2 \lambda_3, \\ 
g_{2,0} & = u_0 u_1^2 \lambda_1 \lambda_3^2, \\ 
g_{3,0} & = 0.  
\end{split}
\end{equation*}
Thus, the coefficient of $B_0^5(t)$ of (\ref{eq:Gi}), is given by 
\[
\sum_{i=0}^3 b_i g_{i,0} = 0.
\]
We can perform similar computations to show that the other coefficients (of $B_{j}^5(t)=\binom{5}{j}t^j(1-t)^{5-j},$ $j=0,\ldots,5$) are all zero, thus proving the theorem.

\subsection{Proof of Lemma \ref{lem:degen}}

\begin{proof}
Using (\ref{eq:t1t}) we can write  
\[
q_2(\mathbf{p}(t)) = \frac{t^2 (1-t)^2}{w(t)^2} (a_0 (1-t)^2 + 2 a_1 t(1-t) + a_2 t^2)
\]
where 
\begin{equation}
\begin{split}
a_0 & = u_0 u_1^2 u_3 \lambda_2^2 - u_0 u_1 u_2^2 \lambda_1\lambda_3 = \Phi_2, \\
a_1 & = u_1^2 u_2^2 \lambda_0\lambda_3 + u_0 u_1 u_2 u_3 \lambda_1\lambda_2 - 2 u_0 u_1 u_2 u_3 \lambda_1\lambda_2 = \frac{1}{2} \Phi_3, \\
a_2 & = u_0 u_2^2 u_3 \lambda_1^2 - u_1^2 u_2 u_3 \lambda_0\lambda_2 = \Phi_1.
\end{split}
\end{equation}
Thus, the terms inside the parentheses are given by $r(t).$  
\end{proof} 


\section{Construction of the basis functions}\label{sec:construct}

Finally we provide a discussion of how the basis functions are constructed. The functions $(K_i)_{i=0}^3$ can be thought of as the implicit representations of various limiting configurations of non-negative weights $w_0, w_1, w_2$ and $w_3.$ The configurations for each of the functions are given in Table \ref{tab:config}. We can use the theory developed earlier in the paper to show that these limits are valid.
\begin{table}
\centering
\begin{tabular}{|l||c|c|c|c|}
\hline 
 Basis function               & $w_0$ & $w_1$ & $w_2$ & $w_3$ \\ \hline\hline
 $K_0 = L_{01} L_{12} L_{23}$ & $1/w$ & $w  $ & $w  $ & $1/w$ \\ \hline
 $K_1 = L_{01} L_{13}^2$      & $1/w$ & $w  $ & $1/w$ & $w  $ \\ \hline
 $K_2 = L_{02}^2 L_{23}$      & $w  $ & $1/w$ & $w  $ & $1/w$ \\ \hline
 $K_3 = L_{03}^3$             & $w  $ & $1/w$ & $1/w$ & $w  $ \\ \hline
\end{tabular}
\caption{Weight configurations for the basis functions $(K_i)_{i=0}^3$ as $w\rightarrow\infty.$}
\label{tab:config}
\end{table}

That the basis functions are given by the limiting configurations can be seen by evaluating the coefficients with their respective weights. For example, consider the function $K_1,$ assuming that no three control points are collinear. The coefficients $(b_i)_{i=0}^3$ for the weights corresponding to $K_1$ in Table \ref{tab:config} are
\begin{eqnarray*}
b_0 & = -(9\lambda_1^2 \lambda_2^2 - 18\Lambda), \\
b_1 & = 9\lambda_1^3 \lambda_3 - 27w^4 \Lambda, \\
b_2 & = 9\lambda_0 \lambda_2^3 - 27/w^4 \Lambda, \\
b_3 & = 9\lambda_0^2 \lambda_3^2 - \Lambda.
\end{eqnarray*} 
We can divide all coefficients by $w^4$ and take the limit as $w\rightarrow\infty$ to see that the limiting coefficients are
\begin{eqnarray*}
\beta_0 = \lim_{w\rightarrow\infty}b_0/w^4 & = 0, \\
\beta_1 = \lim_{w\rightarrow\infty}b_1/w^4 & = -27\Lambda, \\
\beta_2 = \lim_{w\rightarrow\infty}b_2/w^4 & = 0, \\
\beta_3 = \lim_{w\rightarrow\infty}b_3/w^4 & = 0.
\end{eqnarray*} 
Clearly we can divide through by the non-zero constant $-27\Lambda,$ in order to normalize the coefficients. The other basis functions can be 
treated similarly. 

Such limits are more difficult to treat using the parametric representation since there does not exist a single rational cubic B\'ezier 
representation which traverses the algebraic curves which define the basis functions. This is due to all of the functions being reducible 
products of linear forms. However, we can evaluate the parametric limit at certain parameters in order to generate some points along the basis 
functions. Once again we take $K_1$ as an example. For all $t\neq0,$ we have 
\begin{eqnarray*}
&   & \lim_{w\rightarrow\infty, t\neq0}  \frac{\mathbf{c}_0 \frac{1}{w}(1-t)^3 + \mathbf{c}_1 3 w(1-t)^2 t + \mathbf{c}_2 \frac{3}{w}(1-t)t^2 + \mathbf{c}_3 wt^3}{\frac{1}{w}(1-t)^3 + 3w(1-t)^2 t + \frac{3}{w}(1-t)t^2 + wt^3} = \frac{\mathbf{c}_1 3 (1-t)^2 + \mathbf{c}_3 t^2}{3(1-t)^2 + t^2}. 
\end{eqnarray*}
When $t=0$ we can evaluate the parametric description before taking limits. We then get 
\begin{eqnarray*}
\lim_{w\rightarrow\infty, t=0}  \frac{\mathbf{c}_0 \frac{1}{w}(1-t)^3 + \mathbf{c}_1 3 w(1-t)^2 t + \mathbf{c}_2 \frac{3}{w}(1-t)t^2 + \mathbf{c}_3 wt^3}{\frac{1}{w}(1-t)^3 + 3w(1-t)^2 t + \frac{3}{w}(1-t)t^2 + wt^3} = \mathbf{c}_0.
\end{eqnarray*}
So we see that the points between $\mathbf{c}_1$ and $\mathbf{c}_3$ are traversed with quadratic multiplicity for $t\in[-\infty,\infty]\setminus\{0\}$, while the point $\mathbf{c}_0$ is included at $t=0.$ This justifies the representation given by $K_1.$ Once again, the other basis functions can be treated similarly.

\end{appendix}

\end{document}